\documentclass[a4paper, reqno, 11pt]{amsart}

\usepackage[usenames,dvipsnames]{color}
\usepackage{amsthm,amsfonts,amssymb,amsmath,amsxtra, array}
\usepackage[all]{xy}
\usepackage{xr-hyper}
\usepackage{dynkin-diagrams}
\usepackage{verbatim}
\usepackage{etex, tikz-cd, epic}
\usepackage[margin=1.25in]{geometry}
\usepackage{mathrsfs}
\usepackage{bbm}
\usepackage{BOONDOX-calo}
\usepackage{float}
\restylefloat{table}
\usepackage{MnSymbol}
\usepackage{soul}

\RequirePackage{xspace}
\RequirePackage{etoolbox}
\RequirePackage{varwidth}
\RequirePackage[shortlabels]{enumitem}
\RequirePackage{mathtools}
\RequirePackage{longtable}
\def\le{\leqslant}
\def\a{\alpha}

\def\D{\Delta}

\def\s{\sigma}

\def\i{^{-1}}

\def\<{\langle}
\def\>{\rangle}

\newcommand{\fkp}{\ensuremath{\mathfrak{p}}\xspace}

\newcommand{\fkO}{\ensuremath{\mathfrak{O}}\xspace}

\newcommand{\BF}{\ensuremath{\mathbb {F}}\xspace}
\newcommand{{\BG}}{\ensuremath{\mathbb {G}}\xspace}

\newcommand{{\BK}}{\ensuremath{\mathbb {K}}\xspace}

\newcommand{\BN}{\ensuremath{\mathbb {N}}\xspace}

\newcommand{\BQ}{\ensuremath{\mathbb {Q}}\xspace}
\newcommand{\BR}{\ensuremath{\mathbb {R}}\xspace}
\newcommand{\BS}{\ensuremath{\mathbb {S}}\xspace}

\newcommand{\BZ}{\ensuremath{\mathbb {Z}}\xspace}

\newcommand{\CA}{\ensuremath{\mathcal {A}}\xspace}
\newcommand{\CB}{\ensuremath{\mathcal {B}}\xspace}

\newcommand{\CH}{\ensuremath{\mathcal {H}}\xspace}

\newcommand{\CT}{\ensuremath{\mathcal {T}}\xspace}

\newcommand{\CX}{\ensuremath{\mathcal {X}}\xspace}

\newcommand{\tF}{{\widetilde{F}}}

\newcommand{\tH}{{\widetilde{H}}}

\newcommand{\tL}{\widetilde{L}}

\newcommand{\ta}{{\widetilde{a}}}
\newcommand{\tb}{{\widetilde{b}}}
\newcommand{\tc}{{\widetilde{c}}}

\newcommand{\ad}{{\mathrm{ad}}}

\DeclareMathOperator{\Aut}{Aut}

\DeclareMathOperator{\diag}{diag}

\DeclareMathOperator{\Gal}{Gal}
\newcommand{\GL}{\mathrm{GL}}

\DeclareMathOperator{\Nm}{Nm}

\newcommand{\PGL}{{\mathrm{PGL}}}

\newcommand{\red}{\ensuremath{\mathrm{red}}\xspace}

\DeclareMathOperator{\Res}{Res}

\newcommand{\SU}{{\mathrm{SU}}}
\newcommand{\Sc}{{\mathrm{sc}}}

\newcommand{\U}{\mathrm{U}}

\newtheorem{theorem}{Theorem}
\newtheorem{proposition}[theorem]{Proposition}
\newtheorem{lemma}[theorem]{Lemma}

\newtheorem{corollary}[theorem]{Corollary}

\theoremstyle{definition}
\newtheorem{definition}[theorem]{Definition}

\newtheorem{remark}[theorem]{Remark}

\def\af{\text{af}}

\def\der{\text{der}}
\def\pr{\text{pr}}

\def\kk{\mathbf k}

\numberwithin{equation}{section}
\numberwithin{theorem}{section}

\setitemize[0]{leftmargin=*,itemsep=\the\smallskipamount}
\setenumerate[0]{leftmargin=*,itemsep=\the\smallskipamount}

\renewcommand{\to}{%
   \ifbool{@display}{\longrightarrow}{\rightarrow}%
   }
\let\shortmapsto\mapsto
\renewcommand{\mapsto}{%
   \ifbool{@display}{\longmapsto}{\shortmapsto}%
   }
\newlength{\olen}
\newlength{\ulen}
\newlength{\xlen}
\newcommand{\xra}[2][]{%
   \ifbool{@display}%
      {\settowidth{\olen}{$\overset{#2}{\longrightarrow}$}%
       \settowidth{\ulen}{$\underset{#1}{\longrightarrow}$}%
       \settowidth{\xlen}{$\xrightarrow[#1]{#2}$}%
       \ifdimgreater{\olen}{\xlen}%
          {\underset{#1}{\overset{#2}{\longrightarrow}}}%
          {\ifdimgreater{\ulen}{\xlen}%
             {\underset{#1}{\overset{#2}{\longrightarrow}}}
             {\xrightarrow[#1]{#2}}}}%
      {\xrightarrow[#1]{#2}}
   }
\makeatother
\newcommand{\xyra}[2][]{%
   \settowidth{\xlen}{$\xrightarrow[#1]{#2}$}%
   \ifbool{@display}%
      {\settowidth{\olen}{$\overset{#2}{\longrightarrow}$}%
       \settowidth{\ulen}{$\underset{#1}{\longrightarrow}$}%
       \ifdimgreater{\olen}{\xlen}%
          {\mathrel{\xymatrix@M=.12ex@C=3.2ex{\ar[r]^-{#2}_-{#1} &}}}%
          {\ifdimgreater{\ulen}{\xlen}%
             {\mathrel{\xymatrix@M=.12ex@C=3.2ex{\ar[r]^-{#2}_-{#1} &}}}
             {\mathrel{\xymatrix@M=.12ex@C=\the\xlen{\ar[r]^-{#2}_-{#1} &}}}}}%
      {\mathrel{\xymatrix@M=.12ex@C=\the\xlen{\ar[r]^-{#2}_-{#1} &}}}%
   }
\makeatletter
\newcommand{\xla}[2][]{%
   \ifbool{@display}%
      {\settowidth{\olen}{$\overset{#2}{\longleftarrow}$}%
       \settowidth{\ulen}{$\underset{#1}{\longleftarrow}$}%
       \settowidth{\xlen}{$\xleftarrow[#1]{#2}$}%
       \ifdimgreater{\olen}{\xlen}%
          {\underset{#1}{\overset{#2}{\longleftarrow}}}%
          {\ifdimgreater{\ulen}{\xlen}%
             {\underset{#1}{\overset{#2}{\longleftarrow}}}
             {\xleftarrow[#1]{#2}}}}%
      {\xleftarrow[#1]{#2}}
   }
\newcommand{\isoarrow}{%
   \ifbool{@display}{\overset{\sim}{\longrightarrow}}{\xrightarrow\sim}%
   }
 \newcommand{\bF}{{\breve{F}}}
  
\newcommand{\bW}{{\breve{W}}}
\newcommand{\bl}{{\breve{l}}}
\newcommand{\bb}{{\breve{b}}}
\newcommand{\bI}{{\breve{I}}}

\newcommand{\ba}{{\breve{a}}}
\newcommand{\bc}{{\breve{c}}}

\newcommand{\bs}{{\breve{s}}}

\newcommand{\bw}{{\breve{w}}}
\newcommand{\bv}{{\breve{v}}}

\newcommand{\bal}{{\breve{\mathbcal{a}}}}

\newcommand{\bBS}{\breve{\BS}}

\newcommand{\al}{{{\mathbcal{a}}}}

\begin{document}

\title[]{Tits groups of affine Weyl groups}
\author{Radhika Ganapathy}
\address{Department of Mathematics, Indian Institute of Science, Bengaluru, Karnataka  - 560012, India.}
\email{radhikag@iisc.ac.in}


%\date{\today}

\begin{abstract} Let $G$ be a connected, reductive group over a non-archimedean local field $F$. Let $\bF$ be the completion of the maximal unramified extension of $F$ contained in a separable closure $F_s$.  In this article, we construct a Tits group of the affine Weyl group of $G(F)$ when the derived subgroup of $G_\bF$ does not contain a simple factor of unitary type. If $G$ is a quasi-split ramified odd unitary group, we show that there always exist representatives in $G(F)$ of affine simple reflections that satisfy Coxeter relations (which is weaker than asking for the existence of a Tits group). If $G = \U_{2r}, r \geq 3,$ is a quasi-split ramified even unitary group, we show that there don't even exist representatives in $G(F)$ of the affine simple reflections that satisfy Coxeter relations.

\end{abstract}

\maketitle

\bigskip

\section{Introduction}

Let $F$ be a non-archimedean local field and let $G$ be a connected, reductive group over $F$. Let $W_\af$ denote the affine Weyl group of $G(F)$ and let $W$ denote the Iwahori-Weyl group of $G(F)$. In \cite{GH23}, the notion of a Tits group of the Iwahori-Weyl group was introduced. Let us first explain a motivation for studying the question on existence of a Tits group of the Iwahori-Weyl group.
It arises naturally from the question of establishing a variant of the Kazhdan isomorphism for general connected reductive groups. Let us briefly recall this work of Kazhdan. First,  given a local field $F'$ of characteristic $p$ and an integer $m \geq 1$, there exists a local field $F$ of characteristic 0 such that $F'$ is $m$-close to $F$, i.e., $\fkO_F/\fkp_F^m \cong \fkO_{F'}/\fkp_{F'}^m$. Let $G$ be a split, connected reductive group defined over $\BZ$. For an object $X$ associated to the field $F$, we will use the notation $X'$ to denote the corresponding object over $F'$. 
 In \cite{kaz86}, Kazhdan proved that  given $n \geq 1$, there exists $l \geq n$ such that if $F$ and $F'$ are $l$-close, then there is an algebra isomorphism $\mathrm{Kaz}_n:\mathscr{H}(G(F), K_n) \rightarrow \mathscr{H}(G(F'), K_n')$, where $K_n$ is the $n$-th usual congruence subgroup of $G(\fkO_F)$. 
Hence, when the fields $F$ and $F'$ are sufficiently close, we have a bijection
 \begin{align*}
 &\text{\{Iso. classes of irr.  admissible representations $(\Pi, V)$ of $G(F)$ such that $\Pi^{K_n} \neq 0$\}} \\
  &\longleftrightarrow\text{\{Iso. classes of irr. admissible representations  $(\Pi', V')$ of $G(F')$ such that $\Pi'^{K_n'} \neq 0$\}}. 
\end{align*}
Let $I$ be an Iwahori subgroup of $G(F)$ and for $n \in \BN$, let $I_n$ be the $n$-th congruence subgroup of $I$. In \cite{How85}, Howe discovered a nice presentation of $\CH(G(F), I_n)$ when $G=GL_m(F)$. Here the generators are the characteristic functions $\mathbbm{1}_{g I_n}$ for $g \in I/I_n$ and $\mathbbm{1}_{I_n m(w) I_n}$, where $w$ runs over elements of $W$ of length $0$ and $1$, and $m(w)$ is a nice representative of $w$ in $G(F)$. This presentation is a refinement of the Iwahori-Matsumoto presentation and has some nice applications to the representation theory of $p$-adic groups. In more detail, Lemaire (see \cite{Lem01}) used this presentation of Howe and proved that if the fields $F$ and $F'$ are $n$-close, then the Hecke algebras $\mathscr{H}(\GL_m(F), I_n)$ and  $\mathscr{H}(\GL_m(F'), I_n')$ are isomorphic. Hence, when the fields $F$ and $F'$ are $n$-close, we have a bijection
 \begin{align}\label{eq:In}
 &\text{\{Iso. classes of irr.  admissible representations $(\Pi, V)$ of $GL_m(F)$ such that $\Pi^{I_n} \neq 0$\}} \\
  &\longleftrightarrow\text{\{Iso. classes of irr. admissible representations  $(\Pi', V')$ of $\GL_m(F')$ such that $\Pi'^{I_n'} \neq 0$\}}. \nonumber
\end{align}
Lemaire further showed that if $\Pi \leftrightarrow \Pi'$ as in \eqref{eq:In}, and one is generic, then so is the other; this proof relies on the presentation written down by Howe. The results of Howe and Lemaire have been generalized to split, connected, reductive groups, and these results have interesting applications to the local Langlands correspondence (see \cite{Gan15, GV17}). To establish this variant of the Kazhdan isomorphism for general connected reductive groups, it is crucial to choose a nice enough set of representatives of the Iwahori-Weyl group of $G(F)$ that depend on the field $F$ as minimally as possible, so that the ``corresponding" set of representatives can be considered over the field $F'$. This question of choosing a nice enough set of representatives of the elements of the Iwahori-Weyl group leads to the question on the existence of a Tits group of the Iwahori-Weyl group, whose study was initiated in \cite{GH23}. There, the notion of a Tits group $\CT_\af$ of the affine Weyl group $W_\af$ and $\CT$ of the Iwahori-Weyl group $W$ were defined, modelling the case of a finite Weyl group due to Tits (see \cite{Ti}), and the following theorem was proved.
\begin{theorem}[\S 5 and \S 6 of \cite{GH23}]\label{thm:intro0} If  $G$ splits over an unramified extension of $F$, a Tits group $\CT$ of $W$ exists and contains a Tits group $\CT_\af$ of $W_\af$.
\end{theorem}
In \cite{GH23}, it was also proved that the affine Weyl group (and hence also the Iwahori-Weyl group) of the even unitary group $U_6 \subset \Res_{L/F} \GL_6$ where $L/F$ is wildly ramified quadratic extension, does not admit a Tits group. 

In this article, we give a complete answer to the question of the existence of a Tits group of the affine Weyl group. Note that the affine Weyl group (which is the Iwahori-Weyl group of $G(F)$ when $G$ is semisimple and simply connected) is in fact a Coxeter group; let $\BS$ denote its Coxeter generating set whose elements are affine simple reflections. To explain the results in this article, let us quickly recall the definition of $\CT_\af$ from \cite{GH23}. The Tits group $\CT_\af$ (see Definition \ref{def:tits2}) of the affine Weyl group $W_\af$ is a certain subgroup of $G(F)$ that contains representatives $n_s$ of the elements $s \in \BS$, such that
\begin{itemize}
\item we have an exact sequence 
\[ 1 \rightarrow S_2 \rightarrow \CT_\af \rightarrow W_\af \rightarrow 1\] where $S_2$ is a certain elementary abelian 2-group, 
    \item the elements $n_s, s \in \BS$, satisfy Coxeter relations,
    \item the element $n_s^2 = b^\vee(-1)$ lies in $S_2$, where $s = s_a$ for a suitable affine root $a$ and $b$ is the gradient of $a$.
\end{itemize}

Let us summarize the results of this article.  Let $G$ be a connected, reductive group over $F$. Let $\bF$ denote the completion of the maximal unramified extension of $F$ contained in a fixed separable closure $F_s$. First, we consider the question of existence of the Tits group $\breve \CT_\af$ of the affine Weyl group $\bW_\af$ over $\bF$.  Let $\bBS$ be the set of affine simple reflections of $\bW_\af$ which forms a Coxeter generating set of $\bW_\af$. 
Let $G_\bF^{\Sc}$ denote the simply connected cover of the derived subgroup of $G_\bF$.  For a connected, reductive group $H$ over $F$, we say for brevity that $H_\bF$ is of even (resp. odd) unitary type if $H_\bF^\Sc \cong \Res_{L/\bF} \SU_n$ where $\SU_n$ is the special unitary group defined by a ramified quadratic extension $\tF/L$ and with $n$ even (resp. $n$ odd). Similarly, we say $H_\bF$ is of odd orthogonal type if $H_\bF^\Sc \cong \Res_{\tF/\bF} \mathrm{Spin}_{2n+1}$. The following theorem is proved in \S \ref{sec:TitsBF}.
\begin{theorem}\label{thm:intro}

    \begin{enumerate}
        \item If $G_\bF^\der$ does not contain a simple factor of unitary type, then $\breve \CT_\af$ exists.
        \item If $G_\bF = \U_{2r+1}, r \geq 1$, is an odd unitary group, there exist a set of representatives $n_\bs, \bs \in \bBS$ that satisfy Coxeter relations. However, $n_\bs^2$ is not necessarily of order 2 (and does not necessarily lie in $\breve S_2$). 
        \item If $G_\bF = \U_{2r}, r \geq 3$, is an even unitary group, there \textit{do not} exist a set of representatives of elements of $\bBS$ in $G(\bF)$ that satisfy all the Coxeter relations.
    \end{enumerate}
\end{theorem}
Let us briefly explain the idea of the proof. Note that by Steinberg's Theorem (see \cite[Theorem 56]{Ste65}), $G_\bF$ is quasi-split. It admits a Chevalley-Steinberg system and this can be used to construct representatives of elements of $\bBS$ (see \S \ref{sec:repsASR}), whose squares can be calculated explicitly (see \S \ref{sec:squares}). To show that these representatives of elements of $\bBS$ satisfy Coxeter relations, we invoke \cite[Proposition 6.1.8]{BT1}. We show that the hypotheses of \textit{loc.cit.} are satisfied in all but two cases; one when $G_\bF^\der$ contains a simple factor of odd orthogonal type, and other when $G_\bF^\der$ contains a simple factor of even unitary type. If $G_\bF^\der$ does not contain a simple factor of these two types, \cite[Proposition 6.1.8]{BT1} yields the proof of (1) of Theorem \ref{thm:intro}. If $G_\bF^\der$ contains a simple factor of odd orthogonal type, we use \cite[Lemma 3.1]{Gan15} (for this simple factor) to prove (1) of Theorem \ref{thm:intro} (see Corollary \ref{cor:Coxeter} and Theorem \ref{thm:TitsbF}).  Part (2) of Theorem \ref{thm:intro} follows from a calculation carried out in Section \ref{sec:squares}. Let us say a few words about Theorem \ref{thm:intro} (3). As mentioned above, \cite[Proposition 6.1.8]{BT1} cannot be applied to the case when $G$ is a ramified even unitary group because a relevant rank 2 finite root system does not satisfy the hypotheses of the proposition in \textit{loc.cit}. We show in Proposition \ref{prop:eug} that for the unitary group $\U_{2r} \subset \Res_{\tF/\bF} \GL_{2r}, r \geq 3$, there do not exist representatives of the affine simple reflections that satisfy Coxeter relations. 

Next, we descend the construction in Theorem \ref{thm:intro} (1) to $F$.  Let $\sigma$ denote the Frobenius action on $G(\bF)$ such that $G(F) = G(\bF)^\sigma$.  Note that $W_\af = \bW_\af^\sigma$. The following result is proved in \S \ref{sec:descend}.

\begin{proposition}\label{prop:intro}
    Let $G$ be a connected, reductive group over $F$ such that $G_\bF^\der$ does not contain a simple factor of unitary type. Then there exists a $\sigma$-stable set of representatives of elements of $\bBS$ such that the subgroup $\breve\CT_\af$ of $G(\bF)$ generated by these representatives is a Tits group of $\bW_\af$ over $\bF$. Further, the group $\CT_\af :=\breve\CT_\af^\sigma$ is a Tits group of $W_\af$ over $F$.
\end{proposition}
The proof of construction of representatives of elements of $\bBS$ over $\bF$ that are $\sigma$-stable and satisfy Coxeter relations is a simple generalization of \cite[Proposition 6.1]{GH23}.  This then yields a set of representatives of elements of $\BS$ over $F$ that satisfy Coxeter relations This result, combined with Theorem \ref{thm:intro}, proves Proposition \ref{prop:intro} (see Proposition \ref{prop:descent} and Corollary \ref{cor:TitsF}). 

Finally, we make some remarks about the existence of Tits groups for the case unitary groups.  Let $G = \U_{2r+1}, r \geq 1$ be an odd unitary group over $F$ associated to a ramified quadratic extension $L/F$. Then, a  Tits group of its affine Weyl group exists if and only if $L/F$ is tamely ramified (see Remark \ref{rem:tameodd} and \S \ref{rem:Titsodd1}). Let $G = \U_{2r}, r \geq 3$, be a even unitary group over $F$ associated to a ramified quadratic extension $L/F$. We know that the affine Weyl group of $G$ does not admit a Tits group.  Let $G^*$ be the unique non-quasi-split inner form of $G$. Then, a Tits group of the affine Weyl group of $G^*$ always exists (see \S \ref{rem:Titseven}).

\section*{Acknowledgements}
I thank Xuhua He for the numerous discussions during our collaboration in \cite{GH23} that led up to this work and for his encouragement. I thank Dipendra Prasad for the helpful comments and suggestions on a previous draft of this article. I am grateful to the referee for the careful reading and for the many helpful comments. 

\section{Preliminaries}

\subsection{Notation}\label{notation} Let $F$ be a non-archimedean local field with $\fkO_F$ its ring of integers, $\fkp_F$ its maximal ideal, $\varpi_F$ a uniformizer, and $\kk=\BF_q$ its residue field. Let $p$ be the characteristic of $\kk$. Let $\bar F$ be the completion of a separable closure $F_s$ of $F$. Let $\bF$ be the completion of the maximal unramified subextension with valuation ring $\fkO_\bF$ and residue field $\bar\kk$. Note that $\varpi_F$ is also a uniformizer of $\bF$. Let $\Gamma = \Gal(\bar F/F)$ and $\Gamma_0=\Gal(\bar F/\bF)$. Let $val$ denote the additive valuation on $F$ normalized so that $val(F) = \BZ$.  

Let $G$ be a connected, reductive group over $F$. By Steinberg's Theorem (see \cite[Theorem 56]{Ste65}), $G_\bF$ is quasi-split. Let $\sigma$ denote the Frobenius action on $G(\bF)$ such that $G(F) = G(\bF)^\sigma$. Let $A$ be a maximal $F$-split torus of $G$ and $S$ be a maximal $\bF$-split $F$-torus of $G$ containing $A$. Let $T = Z_G(S)$. Then $T$ is defined over $F$ and is a maximal $F$-torus of $G$ containing $S$. Let $\tF$ be the field of invariants of the kernel of the representation of $\Gamma_0$ on $X^*(T)$. This extension is Galois over $\bF$. Hence $T$ and $G$ are split over $\tF$.  

Let $\tilde\Phi(G,T)$ be the set of roots of $T_\tF$ in $G_\tF$. Then the set of relative roots of $S$ in $G_\bF$, denoted by $\breve\Phi(G,S)$, is the set of the restrictions of the elements in $\tilde\Phi(G,T)$ to $S$. Let $\bW_0$ denote the relative Weyl group of $G$ with respect to $S$ and let $W(G,T)$ denote the absolute Weyl group of $G$.

Let $\CB(G, \bF)$ (resp. $\CB(G,F)$) denote the enlarged Bruhat-Tits building of $G(\bF)$ (resp. $G(F)$). Then $\CB(G, \bF)$ carries an action of $\sigma$ and $\CB(G,F)= \CB(G, \bF)^\sigma$. Let $\CA(S, \bF)$ be the apartment in $\CB(G, \bF)$ corresponding to $S$. Let $\bal$ be a $\sigma$-stable alcove in $\CA(S,\bF)$.  Set $\al = \bal^\sigma$; this is an alcove in the apartment $\CA(A, F)$ (see \cite[\S 5.1]{BT2}). 

Let $\bv_0$ be an extra special vertex contained in the closure of $\bal$ (see \cite[Definition 1.3.39 and Proposition 1.3.43]{KP23}). Let $\breve\Phi_\af(G,S)$ denote the set of affine roots of $G(\bF)$ relative to $S$. Let $V = X_*(S) \otimes_\BZ \BR$. The choice of $\bv_0$ also allows us to identify $\CA(S, \bF)$ with $V$ via $\bv_0 \mapsto 0 \in V$, which we now do. We then view $\bal \subset V$. Let $\breve\Delta \subset \breve\Phi_\af(G,S)$ be the set of affine roots such that the corresponding vanishing hyperplanes form the walls of $\bal$. The Weyl chamber in $V$ that contains $\bal$ then yields a set of simple roots for $\breve\Phi(G,S)$ which we denote as $\breve\Delta_0$. Clearly $\breve\Delta_0 \subset \breve\Delta$.

\subsection{Affine Weyl group over $\breve F$}

Let $\bI$ be the Iwahori subgroup associated to $\bal$. Let $\pr:X_*(T) \rightarrow X_*(T)_{\Gamma_0}$ be the natural projection and let  $\kappa_{T, \bF}: T(\bF) \rightarrow X_*(T)_{\Gamma_0}$ denote the Kottwitz homomorphism (see \cite[\S 7.2]{Kot97}). The map $\kappa_{T, \bF}$ is surjective  and its kernel $T(\bF)_1$ is the unique parahoric subgroup of $T(\bF)$. Let $\bW = N_G(S)(\bF)/T(\bF)_1$ be the Iwahori-Weyl group of $G(\bF)$. This group fits into an exact sequence
\[1 \rightarrow X_*(T)_{\Gamma_0} \rightarrow \bW \rightarrow \bW_0 \rightarrow 1.\]

Recall that we have chosen an extra special vertex $\bv_0$. With this, we have a semi-direct product decomposition
\begin{align}\label{decomp1}
    \bW\cong X_*(T)_{\Gamma_0} \rtimes \bW_0.
\end{align}

Let $\bBS = \{s_\ba\;|\; \ba \in \breve\Delta\}$ be the set of simple reflections with respect to the walls of $\bal$. Let $\bBS_0 =\{s_\ba\;|\; \ba \in \breve\Delta_0\}$. Let $\bW_{\af} \subset \bW$ be the Coxeter group generated by $\bBS$. Let $T_{\Sc}, N_{\Sc}$ denote the inverse images of $T\cap G_\der$, resp. $N_
G(S)\cap G_\der$ in $G_\Sc$. Let $S_{\Sc}$ denote the $\bF$-split component of $T_{\Sc}$. Then $\bW_{\af}$ may be identified with the Iwahori-Weyl group of $G_{\Sc}$. It fits into the exact sequence
\begin{align}\label{IWGSDP}
    1 \rightarrow \bW_{\af} \rightarrow \bW \rightarrow X^*(Z(\hat G)^{\Gamma_0}) \rightarrow 1.
\end{align}
Let $\Omega_{\bal}$ be the stabilizer of $\bal$ in $\bW$. Then $\Omega_{\bal}$ maps isomorphically to $X^*(Z(\hat G)^{\Gamma_0})$ and we have a $\sigma$-equivariant semi-direct product decomposition
\[\bW \cong \bW_{\af} \rtimes \Omega_\bal.\]

Let $\bl$ be the length function on $\bW$. Then $\bl(s) = 1$ for all $s \in \bBS$ and $\Omega_\bal$ is the set of elements of length 0 in $\bW$.

\subsection{Affine Weyl group over $F$}\label{sec:WF}
Let $I$ be the Iwahori subgroup of $G(F)$ associated to $\al$. Then $I = \bI^\sigma$. Let $M = Z_G(A)$ and $M(F)_1$ be the unique parahoric subgroup of $M(F)$. We may identify $M(F)_1$  with the kernel of the Kottwitz homomorphism $M(F) \rightarrow X^*(Z(\hat M)^{\Gamma_0})^\sigma$. Let $W= N_G(A)(F)/M(F)_1$ denote the Iwahori-Weyl group of $G(F)$ with length function $l$. 

By \cite[Lemma 1.6]{Ri}, we have a natural isomorphism $W \cong \breve W \,^{\s}$. It is proved in \cite[Proposition 1.11 \& sublemma 1.12]{Ri} that 

(a) for $w, w' \in W$, $\breve \ell(w w')=\breve \ell(w)+\breve \ell(w')$ if and only if $\ell(w w')=\ell(w)+\ell(w')$.

The semi-direct product decomposition of $\bW$ in \eqref{IWGSDP} is $\sigma$-equivariant and yields a decomposition
\[W \cong \bW_{\af}^\sigma \rtimes \Omega_{\bal}^\sigma.\]
Let $W_{\af} = \bW_{\af}^\sigma$ and let $\BS$ be the set of reflections through the walls of $\al$. Then $(W_\af, \BS)$ is a Coxeter system. The group $\Omega_\al$, which is the stabilizer of the alcove $\al$, is isomorphic to $\Omega_{\bal}^\sigma$ and is the set of length 0 elements is $W$. 

The simple reflections $\BS$ of $W_{\af}$ are certain elements in $\bW_{\af}$. The explicit description is as follows. For any $\s$-orbit $\CX$ of $\bBS$, we denote by $\bW_{\CX}$ the parabolic subgroup of $\bW_{\af}$ generated by the simple reflections in $\CX$. If moreover, $\bW_{\CX}$ is finite, we denote by $\bw_{\CX}$ the longest element in $\bW_{\CX}$. It is proved by Lusztig \cite[Theorem A.8]{Lusz} that there exists a natural bijection $s \mapsto \CX$ from $\BS$ to the set of $\s$-orbits of $\bBS$ with $\bW_{\CX}$ finite such that the element $s \in W_{\af} \subset \bW_{\af}$ is equal to $\bw_{\CX}$. 

Let $\Delta \subset \Phi_\af(G,A)$ be the set of affine roots such that the corresponding vanishing hyperplanes form the walls of $\al$.

\section{Definition of a Tits group of the affine Weyl group}\label{tits}

\subsection{A Tits group of the affine Weyl group over $\breve F$}\label{sec:tits} The notion of a  Tits group of the affine and Iwahori-Weyl group was  studied in \cite{GH23}. We recall this definition for affine Weyl groups.

Note that any element $\bw \in \bW_\af$ can be written as $\bw=\bs_{i_1} \cdots \bs_{i_n}$, where $\bs_{i_1}, \cdots, \bs_{i_n} \in \breve \BS$. If $n=\breve \ell(\bw)$, then we say that $\bw=\bs_{i_1} \cdots \bs_{i_n} $ is a reduced expression of $\bw$ in $\bW_\af$.

\begin{definition}\label{def:tits}
Let $\breve S_2$ be the elementary abelian two-group generated by $\breve b^\vee(-1)$ for $\breve b \in \breve \Phi(G, S)$. A {\it Tits group} of $\bW_\af$ is a subgroup $\breve \CT_\af$ of $N_G(S)(\breve F)$ such that 
\begin{enumerate}
    \item The natural projection $\breve \phi: N_G(S)(\breve F) \to \bW_\af$ induces a short exact sequence $$1 \to \breve S_2 \to \breve \CT_\af \xrightarrow{\breve \phi} \bW_\af \to 1.$$
    
    \item There exists a {\it Tits cross-section} $\{m(\bw)\}_{\bw \in \bW}$ of $\bW$ in $\breve \CT$ such that 
    
    \begin{enumerate}
        \item for $\ba \in \breve\Delta$, $m(\bs_{\ba})^2 = \bb^\vee(-1)$, where $\bb$ is the gradient of $\ba$.
        
        \item for any reduced expression $\bw=\bs_{i_1} \cdots \bs_{i_n} $ in $\bW_\af$, we have $m(\bw)=m(\bs_{i_1}) \cdots m(\bs_{i_n})$. 
    \end{enumerate}
\end{enumerate}
\end{definition}

It is easy to see that the condition (2) (b) in Definition \ref{def:tits} is equivalent to 

Condition (2)(b)$^{\dagger}$: $m(\bw \bw')=m(\bw) m(\bw')$ for any $\bw, \bw' \in \bW_{\af}$ with $\bl(\bw\bw') = \bl(\bw) +\bl(\bw')$.

\subsection{A Tits group of the affine Weyl group over $F$}\label{sec:TitsF}
The Tits group of $W_\af$ is defined as follows.  

Note that any element $w \in W_\af$ can be written as $w=s_{i_1} \cdots s_{i_n}$, where $s_{i_1}, \cdots, s_{i_n} \in \BS$. If $n=\ell(w)$, then we say that $w=s_{i_1} \cdots s_{i_n} $ is a reduced expression of $w$ in $W_\af$. 

\begin{definition}\label{def:tits2}
Let $S_2=\breve S_2^\s$. A {\it Tits group} of $W_\af$ is a subgroup $\CT_\af$ of $N_G(A)(F)$ such that 
\begin{enumerate}
    \item The natural projection $\phi: N_G(A)(F) \to W_\af$ induces a short exact sequence $$1 \to S_2 \to \CT_\af \xrightarrow{\phi} W_\af\to 1.$$
    
    \item There exists a {\it Tits cross-section} $\{m(w)\}_{w \in W_\af}$ of $W_\af$ in $\CT_\af$ such that 
    
    \begin{enumerate}
       \item for $a \in \Delta$, $m(s_{a})^2 = b^\vee(-1)$, where $b$ is the gradient of $a$.
        
       \item for any reduced expression $w=s_{i_1} \cdots s_{i_n}$ in $W_\af$, we have $m(w)=m(s_{i_1}) \cdots m(s_{i_n})$. 
    \end{enumerate}
\end{enumerate}
\end{definition}

\section{A Tits group of the affine Weyl group over $\bF$}\label{sec:TitsBF}

\subsection{A set of representatives for affine simple reflections}\label{sec:repsASR}
Let $G$ be a connected, reductive group over $F$. In \cite[\S 5.1]{GH23}, we proved the existence of the Tits group of a finite relative Weyl group of $G_\bF$. We briefly recall some definitions here.

\subsubsection{}\label{Sec:Affpin} Recall that we have chosen an extra special vertex $\breve v_0 \in \CA(S, \bF)$ and we have $\breve\Delta_0 \subset \breve \Delta$.  This determines a Borel subgroup of $G$ defined over $\bF$. Let $\tilde\Delta_0$ be the set of simple roots of $\tilde\Phi(G,T)$ such that  their restrictions to $S$ lie in $\breve\Delta_0$ provided the restriction is non-zero.  We consider a Steinberg pinning $(x_\ta)_{\ta \in \tilde\Delta_0}$ of $G_\tF$ relative to $S$ (see \cite[\S 4.1.3]{BT2}), which extends to a Chevalley-Steinberg system  $x_\ta: \BG_a \xrightarrow{\cong} U_\ta$ for all $\ta \in \tilde\Phi(G,T)$ and is compatible with the action of $\Gal(\tF/\bF)$. The valuation attached to this Chevalley-Steinberg system determines a point in $\CA(S,\bF)$, which we require is $\breve v_0$.  

From this, we get a set of pinnings for $\ba \in \breve\Phi(G,S)$, which we briefly recall. Let $U_\ba$ be the root subgroup of the root $\ba$. Let $\tilde a \in \tilde\Phi(G,T)$ be such that $\tilde a|_S = \ba$. Let $\bF_\ta$ be the subfield of $\tF$ corresponding to the stabilizer of $\ta$ in $\Gal(\tF/\bF)$.  When $2\ba$ is not a root, we have $x_\ba: \Res_{\bF_{\ta}/\bF}\BG_a \xrightarrow{\cong}   U_\ba$. 
If $2\ba$ is a root, let $\ta, \ta' \in \tilde\Phi(G,T)$ such that $\ta, \ta'|_S =\ba$ and $\ta+\ta'$ is a root. let $H_0(\bF_\ta, \bF_{\ta+\ta'}) = \{(u,v) \in \bF_\ta \times \bF_\ta\;|\; u \cdot \gamma_0(u) = v+\gamma_0(v)\}$, where $\gamma_0$ is the non-trivial $\bF_{\ta+\ta'}$-automorphism of $\bF_\ta$. We have $$x_\ba:  \Res_{\bF_{\ta+\ta'}/\bF} H_0(\bF_\ta, \bF_{\ta+\ta'})\xrightarrow \cong U_\ba.$$

For any $\ta \in \tilde\Delta_0$, let 
\begin{align}\label{tilderep}
n_{s_\ta}:=x_\ta(1) x_{-\ta}(1) x_{\ta}(1).
\end{align}
Here we use the convention of \cite[\S 3.2.1 and \S 4.1.5]{BT2}, and $n_{s_\ta} \in N_G(S)(\bF)$ (This is different from the convention used in \cite{Spr09} where $n_{s_\ta}:=x_\ta(1) x_{-\ta}(-1) x_{\ta}(1)$).

\subsubsection{} Let $\ba \in \breve \D_0$ be such that $2\ba$ is not a root. Set
\begin{align}\label{Rep1}
n_{s_\ba} := x_\ba(1)x_{-\ba}(1)x_\ba(1).
\end{align}

 Let $\ba \in \breve \D_0$ be such that $2\ba$ is a root. By \cite[ Chapter V, \S 4,  Proposition 7]{Ser79}, there exists $c \in \bF_\ba$ such that $c \gamma_0(c)=2$. Set \begin{equation}\label{Rep2'}
    n_{s_\ba} = x_\ba(c,1) x_{-\ba}(c,1)x_\ba(c,1).
\end{equation}

By \cite[\S 4.1.11]{BT2}, we have 
\begin{align}\label{Rep2}
n_{s_\ba}= \prod n_{s_\ta}^{-1} n_{s_{\ta'}}n_{s_{\ta}}^{-1}.
\end{align}
where the product is indexed by the family of sets $\{\ta, \ta'\}$ with $\ta, \ta' \in \tilde\Phi(G,T)$ such that $\ta+\ta'$ is a root and $\ta|_S =\ta'|_S = \ba$.

A direct calculation shows that 
\begin{align}\label{eq:rfsr2}
    n_{s_\ba}^2 =\begin{cases} \ba^\vee(-1), & \text{ if $2\ba$ is not a root}; \\ 1, & \text{ if $2\ba$ is a root}.\end{cases}
\end{align}
Note that if $2\ba$ is a root, then $\ba^\vee/2$ is a coroot, and $\ba^\vee(-1) = (\frac{\ba^\vee}{2}(-1))^2 = 1$, so the formula above simply says $n_{s_\ba}^2 = \ba^\vee(-1)$ whether or not $2\ba$ is a root.
We introduce the following notation. For a root $\bb \in \Phi(G,S)$, define
\begin{align*}
    \bb_* = \begin{cases}
        \frac{\bb}{2} & \text{ if $\frac{\bb}{2}$ is a root}\\[5pt]
        \bb & \text{ otherwise. }
    \end{cases}
\end{align*}
\subsubsection{}\label{sec:rasr}  Let $\ba \in \breve\Delta \backslash \breve\Delta_0$ and let $\bb \in \breve\Phi(G,S)$ be the gradient of $\ba$. Note that ${s_\bb} = s_{\bb_*}$. Define $n_{s_\bb}$ using the morphisms $x_{\pm\bb_*}$ as above. Note that $\bb^\vee, \bb_*^\vee  \in X_*(T_{\Sc})_{\Gamma_0} \hookrightarrow X_*(T)_{\Gamma_0}$. Let $\tb_* \in X^*(T_{\Sc})$ be such that $\tb_*|_S = \bb_*$. Then its dual $\tb_*^\vee \in X_*(T_{\Sc})$ and  $\pr(\tb_*^\vee) = \bb^\vee$. Let  $\varpi_\tF$ be a uniformizer of $\tF$. Define $n_{\bb^\vee} = \Nm_{\tF/\bF} \tb_*^\vee(\varpi_\tF)$. Then $n_{\bb^\vee} \in T_\Sc(\bF)$. Note that  $n_{\bb^\vee}= \Nm_{\bF_{\tb_*}/\bF} \tb_*^\vee(\varpi_{\bF_{\tb_*}})$  where $\varpi_{\bF_{\tb_*}} = \Nm_{\tF/\bF_{\tb_*}} \varpi_\tF $. Define \begin{align}\label{def:affinerep}
    n_{s_\ba} = n_{\bb^\vee} \cdot  n_{s_{\bb}} 
\end{align}

\subsection{Some lemmata on Coxeter relations} The starting point for finding representatives of affine simple reflections that satisfy Coxeter relations is the following proposition  of Bruhat-Tits.

By \cite[\S 6.1.3]{BT1}, for any $u \in U_{-\bb}(\bF)\backslash\{1\}$ and $u' \in U_{-\bb'}(\bF)\backslash\{1\}$, there exists unique triples $(u_1, m(u), u_2) \in U_{\bb}(\bF) \times N_G(S)(\bF) \times U_{\bb}(\bF)$ and $(u'_1, m(u'), u'_2) \in U_{\bb'}(\bF) \times N_G(S)(\bF) \times U_{\bb'}(\bF)$ such that $u = u_1m(u)u_2$ and $u'=u'_1 m(u') u'_2$.

\begin{proposition}[Proposition 6.1.8 of \cite{BT1}]\label{Prop:BT1Cox} Suppose $\Phi$ is a rank 2 root system. Let $\bb^{(1)}, \cdots, \bb^{(2k)}$ be elements of $\Phi^{\red}$ arranged in circular order, that is, with the property that
\begin{align}
    \Phi^\red \cap (\BQ_+ b^{(i-1)}+\BQ_+ b^{(i+1)})=\{b^{(i-1)}, b^{(i)},  b^{(i+1)}\},\; 1<i<2k.
\end{align}
Then, the elements $b = b^{(1)}$ and $b'= b^{(k)}$ form a basis for $\Phi$. Let $u \in U_b$ and $u' \in U_{b'}$ and let $m = m(u)$ and $m' = m(u')^{-1}$. Then
\[mm'\cdots = m'm \cdots\]
where each side of the expression has $k$-factors. 
    \end{proposition}

We begin with the following lemma.
\begin{lemma}\label{lem:rep} Let $\ba \in \breve\Delta \backslash \breve\Delta_0$ and let $\bb \in \breve\Phi(G,S)$ be the gradient of $\ba$. The element $n_{s_\ba}$ defined in \eqref{def:affinerep} lies in $N_G(S)(\bF)$ and the image of $n_{s_\ba}$ in $\bW_\af$ is $s_\ba$.
\end{lemma}
\begin{proof}

Let $\varpi_{\bF_{\tb_*}} = \Nm_{\tF/\bF_{\tb_*}} \varpi_\tF$. If $2\bb_*$ is a root, let $c_0 \in \bF_{\tb_*}$ be such that $c_0 \gamma_0(c_0) = \varpi_{\bF_{\tb_*}} + \gamma_0(\varpi_{\bF_{\tb_*}})$. Such a $c_0$ exists by \cite[ Chapter V, \S 4,  Proposition 7]{Ser79}. 

There exists $u \in U_{\bb_*}$ such that $n_{s_\ba} =m(u)$ for a suitable $u \in U_{\bb_*}$. In fact, if $2\bb_*$ is not a root, then \begin{align}\label{eq:nd}
    n_{s_\ba} = x_{\bb_*}(\varpi_{\bF_{\tb_*}})x_{-{\bb_*}}(\varpi_{\bF_{\tb_*}}^{-1}) x_{\bb_*}(\varpi_{\bF_{\tb_*}}) = x_{-{\bb_*}}(\varpi_{\bF_{\tb_*}}^{-1})x_{\bb_*}(\varpi_{\bF_{\tb_*}}) x_{-{\bb_*}}(\varpi_{\bF_{\tb_*}}^{-1}).
\end{align}
If $2\bb_*$ is a root, then (note that $\bb^*$ is a negative root)
\begin{align}\label{eq:d}
    n_{s_\ba} = x_{-\bb_*}(c_0 \gamma_0(\varpi_{\bF_{\tb_*}}), \varpi_{\bF_{\tb_*}}) x_{\bb_*}(c_0, \gamma_0(\varpi_{\bF_{\tb_*}})^{-1}) x_{-\bb_*}(c_0 \varpi_{\bF_{\tb_*}}, \varpi_{\bF_{\tb_*}}). 
\end{align}

Let $\Gamma_\bb, \Gamma_\bb'$ be the value sets of the root $\bb$ as in \cite[\S 4.2.21]{BT2}. By \cite[Theorem 4.2.22]{BT2}, the affine roots with gradient $\bb$ are of the form $\bb+k, k \in \Gamma_\bb'$.  First, let us deal with the case when $2\bb_*$ is not a root. In this case, 

\[\Gamma_\bb = \Gamma_\bb' =  \frac{1}{e_\bb}\BZ\] where $e_\bb = [\bF_\tb:\bF]$. Since the affine simple root $\ba$ is one of the bounding hyperplanes of the alcove $\bal$ and since $\bb+k$ is an affine root for each $k \in \Gamma_\bb'$, we see that $s_\ba = s_{\bb+\frac{1}{e_\bb}}$. With $M_{\bb,k}$ as in \cite[Page 117]{BT1}, \eqref{eq:nd} and Equation (1) in \cite[\S 4.2.2]{BT2} imply that $n_{s_\ba}$ lies in $M_{\bb, \frac{1}{e_\bb}}$. Then \cite[Proposition 6.2.10]{BT1} implies that the image of $n_{s_\ba}$ in $\bW_\af$ is $s_\ba$.

Next, suppose $2\bb_*$ is a root. In this case, let $e_{\bb_*} = [\bF_{\tb_*}:\bF]$ and $e_{2\bb_*} = [\bF_{\tb_*+\tb_*'}:\bF]$. Note that $e_{\bb_*} = 2e_{2\bb_*}$. As in \cite[Lemma 4.3.3]{BT2}, let $\bF_{\tb_*} = \bF_{\tb_*+\tb_*'}[t]$ where $t^2 - \alpha t +\beta=0$. Then by \cite[\S 4.3.4]{BT2},
    \begin{align}\label{GammabF}\Gamma_{\bb_*}' = \begin{cases} 
  \frac{1}{e_{\bb_*}}\BZ & \text{if  $\alpha=0$}\\
 \frac{1}{2e_{\bb_*}} +\frac{1}{e_{\bb_*}}\BZ & \text{if $\alpha\neq 0$}
 \end{cases} \end{align}
and
 \begin{align}\label{Gamma2bF}\Gamma_{2\bb_*}' = \begin{cases} 
  \frac{1}{e_{\bb_*}} +\frac{2}{e_{\bb_*}}\BZ & \text{if  $\alpha=0$}\\
 \frac{2}{e_{\bb_*}}\BZ & \text{if $\alpha\neq 0$}
 \end{cases} \end{align}

When $\alpha =0$, $\bb_* +\frac{1}{e_{\bb_*}}$ and  $2\bb_* +\frac{1}{e_{\bb_*}}$ are both affine roots. When $\alpha \neq 0$,  $\bb_* +\frac{1}{2e_{\bb_*}}$ and  $2\bb_* +\frac{2}{e_{\bb_*}}$ are both affine roots. Since the affine simple root $\ba$ yields one of the bounding hyperplanes of the alcove $\bal$, we see that $s_\ba = s_{\bb_* +\frac{1}{2e_{\bb_*}}}$. 
Again,  with $M_{\bb_*,k}$ as in \cite[Page 117]{BT1}, \eqref{eq:d} and Equation (1) in \cite[\S 4.2.2]{BT2} imply that $n_{s_\ba}$ lies in $M_{\bb_*, \frac{1}{2e_{\bb_*}}}$. Then \cite[Proposition 6.2.10]{BT1} again implies that the image of $n_{s_\ba}$ in $\bW_\af$ is $s_\ba$.
\end{proof}

\begin{lemma}\label{lem:CR}
    Let $\ba , \ba' \in \breve\Delta$ and let $\bb$ and $ \bb'$ be the gradients of $\ba$ and $\ba'$ respectively. Consider the root system $\Phi_{\bb, \bb'} = \langle \bb, \bb' \rangle$ consisting of elements of $\breve\Phi(G,S)$ that lie in the $\BQ$-span of $\bb$ and $\bb'$.  If the elements $\bb^{(1)}, \cdots, \bb^{(2k)}$ of $\Phi^\red_{\bb, \bb'}$ can be arranged in circular order so that $\bb^{(1)} = \bb$ and $\bb^{(k)} = \bb'$, then the elements $n_{s_\ba}$ and $n_{s_\ba'}$ satisfy Coxeter relations. 
\end{lemma}
\begin{proof}

    In view of Proposition \ref{Prop:BT1Cox} and the hypothesis of the lemma, it suffices to see that $n_{s_\ba}$ and $n_{s_\ba'}$ are of the form $m(u)$ and $m(u')$ for suitable $u \in U_{\bb_*}$ and $u'\in U_{\bb_*'}$. If $\ba, \ba' \in \breve\Delta_0$, then this is clear from \eqref{Rep1} and \eqref{Rep2}. If $\ba \in \breve\Delta \backslash \breve\Delta_0$, then  with $\bb$ the gradient of $\ba$, this was observed in equations \eqref{eq:nd} and \eqref{eq:d} of Lemma \ref{lem:rep}. 
\end{proof}

\subsection{On circular ordering}\label{sec:CO}
Now, let $\ba , \ba' \in \breve\Delta$ and let $\bb$ and $ \bb'$ be the gradients of $\ba$ and $\ba'$ respectively. In order to prove that $n_{s_\ba}$ and $n_{s_{\ba'}}$ satisfy Coxeter relations, we consider the root system $\Phi_{\bb, \bb'}$ spanned by $\bb$ and $\bb'$, and see if the elements $\bb^{(1)}, \cdots, \bb^{(2k)}$ of $\Phi_{\bb, \bb'}^\red$ can be arranged in circular order so that $\bb^{(1)} = \bb$ and $\bb^{(k)} = \bb'$. Recall that $\Phi_{\bb, \bb'} = \breve\Phi(G,S) \cap \BQ\langle \bb, \bb'\rangle$  is the closed subsystem of $\breve\Phi(G,S)$. If $\ba, \ba' \in \breve\Delta_0$, then this can be done, as already observed in \cite[Proposition IV.6]{AHHV}. So we assume that $\ba' \in \breve\Delta_0$ and $\ba \in \breve\Delta \backslash \breve\Delta_0$. We go through the table given in \cite[\S 1.4.6]{BT1} case by case, and follow the notation of \cite[Plate I - VII]{Bou02} for the roots.  
\begin{enumerate}
    \item[$A_n$:]  If $\ba' = \a_i$ for $i \neq 1, n$, then $\Phi_{\bb, \bb'}$ is a product of rank 1 root systems. If $\ba' = \a_n$, then $\Phi_{\bb, \bb'} = \langle \epsilon_{n+1}-\epsilon_1, \epsilon_n - \epsilon_{n+1} \rangle$ and the circular order is given by $\epsilon_{n+1}-\epsilon_1, \epsilon_n-\epsilon_1, \epsilon_n - \epsilon_{n+1}, \epsilon_1 - \epsilon_{n+1}, \epsilon_1 - \epsilon_n, \epsilon_{n+1} - \epsilon_n$. If $\ba' = a_1$, then $\Phi_{\bb, \bb'} = \langle \epsilon_{n+1}-\epsilon_1, \epsilon_1 - \epsilon_{2} \rangle$ and the circular order is given by $\epsilon_{n+1}-\epsilon_1, \epsilon_{n+1}-\epsilon_2, \epsilon_1 - \epsilon_{2}, \epsilon_1 - \epsilon_{n+1}, \epsilon_2 - \epsilon_{n+1}, \epsilon_{2} - \epsilon_1$.
    \item[$B_n$:]  If $\ba' = \a_i$ for $i>2$, then $\Phi_{\bb, \bb'}$ is a product of rank 1 root systems. If $\ba' = \a_2$, then $\Phi_{\bb, \bb'} = \langle -\epsilon_1-\epsilon_2, \epsilon_2 - \epsilon_3 \rangle$ and the circular order is given by $ -\epsilon_1-\epsilon_2, -\epsilon_1-\epsilon_3,\epsilon_2-\epsilon_3,\epsilon_1+\epsilon_2,\epsilon_1+\epsilon_3,\epsilon_3-\epsilon_2$.  If $\ba' = \a_1$, then $\Phi_{\bb, \bb'} = \langle -\epsilon_1 - \epsilon_2, \epsilon_1 - \epsilon_2 \rangle$ which generates an irreducible rank 2 root system with 8 roots. These are $\{\pm(\epsilon_1 \pm \epsilon_2), \pm \epsilon_1, \pm \epsilon_2\}$. These roots \textbf{cannot be put in circular order with $\bb^{(1)} = \bb = -\epsilon_1-\epsilon_2$ and $\bb^{(4)} = \bb' = \epsilon_1 -\epsilon_2$}. 
    \item[$B - C_n$:] If $\ba' = \a_i$ for $i>2$, then $\Phi_{\bb, \bb'}$ is a product of rank 1 root systems. If $\ba' = \a_2$, then $\Phi_{\bb, \bb'} = \langle -\epsilon_1-\epsilon_2, \epsilon_2 - \epsilon_3 \rangle$ and the circular order is given by $ -\epsilon_1-\epsilon_2, -\epsilon_1-\epsilon_3,\epsilon_2-\epsilon_3,\epsilon_1+\epsilon_2,\epsilon_1+\epsilon_3,\epsilon_3-\epsilon_2$. If $\ba' = \a_1$, then $\Phi_{\bb, \bb'} = \langle -\epsilon_1 - \epsilon_2, \epsilon_1 - \epsilon_2 \rangle$ which generates an irreducible rank 2 root system with 8 roots. These are $\{\pm(\epsilon_1 \pm \epsilon_2), \pm 2\epsilon_1, \pm 2\epsilon_2\}$. These roots \textbf{cannot be put in circular order with $\bb^{(1)} = \bb = -\epsilon_1-\epsilon_2$ and $\bb^{(4)} = \bb' = \epsilon_1 -\epsilon_2$}. 
    \item[$C_n$:]  If $\ba' = \a_i$ for $i\geq 2$, then $\Phi_{\bb, \bb'}$ is a product of rank 1 root systems. If $\ba' = \a_1$, then $\Phi_{\bb, \bb'} = \langle -2\epsilon_1, \epsilon_1 - \epsilon_2 \rangle$ and the circular order is given by $ -2\epsilon_1, -\epsilon_1-\epsilon_2,-2\epsilon_2,\epsilon_1 - \epsilon_2, 2\epsilon_1, \epsilon_1+\epsilon_2,2\epsilon_2, \epsilon_2 - \epsilon_1$.
    \item[$C - B_n$:]  If $\ba' = \a_i$ for $i\geq 2$, then $\Phi_{\bb, \bb'}$ is a product of rank 1 root systems. If $\ba' = \a_1$, then $\Phi_{\bb, \bb'} = \langle -\epsilon_1, \epsilon_1 - \epsilon_2 \rangle$ and the circular order is given by $ -\epsilon_1, -\epsilon_1-\epsilon_2,-\epsilon_2,\epsilon_1 - \epsilon_2, \epsilon_1, \epsilon_1+\epsilon_2,\epsilon_2, \epsilon_2 - \epsilon_1$.
    \item[$C - BC_n^{III}$:]  If $\ba' = \a_i$ for $i\geq 2$, then $\Phi_{\bb, \bb'}$ is a product of rank 1 root systems. If $\ba' = \a_1$, then $\Phi_{\bb, \bb'} = \langle -2\epsilon_1, \epsilon_1 - \epsilon_2 \rangle$, $\Phi_{\bb, \bb'}^\red = \langle -\epsilon_1, \epsilon_1 - \epsilon_2 \rangle$ and the circular order is given by $ -\epsilon_1, -\epsilon_1-\epsilon_2,-\epsilon_2,\epsilon_1 - \epsilon_2, \epsilon_1, \epsilon_1+\epsilon_2,\epsilon_2, \epsilon_2 - \epsilon_1$.
    \item[$D_n$:] If $\ba' = \a_i$ for $i=1$ or $i>2$, then $\Phi_{\bb, \bb'}$ is a product of rank 1 root systems. If $\ba' = \a_2$, then $\Phi_{\bb, \bb'} = \langle -\epsilon_1-\epsilon_2, \epsilon_2 - \epsilon_3 \rangle$ and the circular order is given by $ -\epsilon_1-\epsilon_2, -\epsilon_1-\epsilon_3,\epsilon_2-\epsilon_3,\epsilon_1+\epsilon_2,\epsilon_1+\epsilon_3,\epsilon_3-\epsilon_2$.
    \item[$E_6$:] If $\ba' = \a_i$ for $i \neq 2$, then $\Phi_{\bb, \bb'}$ is a product of rank 1 root systems. If $\ba' = \a_2 = \epsilon_1+\epsilon_2$ and $\bb = -\frac{1}{2}(\epsilon_1+\cdots +\epsilon_5 - \epsilon_6-\epsilon_7+\epsilon_8)$, then $\Phi_{\bb, \bb'} = \langle \bb, \bb'\rangle$ and $\bb +\bb'$ is a negative root, whose negative is  represented as 
    
    \begin{align*}
        1\; 2 \; &3\; 2\; 1\\
        &1
    \end{align*} in \cite[Plate V]{Bou02}. The circular order is given by $ \bb, \bb+\bb', \bb', -\bb, -(\bb+\bb'), -\bb' $.
    \item[$E_7$:] If $\ba' = \a_i$ for $i > 1$, then $\Phi_{\bb, \bb'}$ is a product of rank 1 root systems. If $\ba' = \a_1 = \frac{1}{2}(\epsilon_1+\epsilon_8) - \frac{1}{2}(\epsilon_2+\epsilon_3 +\cdots +\epsilon_7)$ and $\bb = - (\epsilon_8 - \epsilon_7)$, then then $\Phi_{\bb, \bb'} = \langle \bb, \bb'\rangle$ and $\bb +\bb'$ is a negative root, whose negative is  represented as 
    
    \begin{align*}
        1\; 3 \; &4\; 3\; 2\; 1\\
        &2
    \end{align*} in \cite[Plate VI]{Bou02}. The circular order is given by $ \bb, \bb+\bb', \bb', -\bb, -(\bb+\bb'), -\bb'.$ 
    \item[$E_8$:] If $\ba' = \a_i$ for $i \neq 8$, then $\Phi_{\bb, \bb'}$ is a product of rank 1 root systems. If $\ba' = \a_8 = \epsilon_7 - \epsilon_6$ and $\bb = - (\epsilon_7 + \epsilon_8)$, then then $\Phi_{\bb, \bb'} = \langle \bb, \bb'\rangle$ and $\bb +\bb'$ is a negative root, whose negative is  represented as 
    
    \begin{align*}
        2\; 4 \; &6\; 5\; 4\; 3\; 1\\
        &3
    \end{align*} in \cite[Plate VII]{Bou02}. The circular order is given by $ \bb, \bb+\bb', \bb', -\bb, -(\bb+\bb'), -\bb'.$ 
    \item[$F_4$:]  If $\ba' = \a_i$ for $i >1$, then $\Phi_{\bb, \bb'}$ is a product of rank 1 root systems. If $\ba' = \a_1$, then $\Phi_{\bb, \bb'} = \langle -\epsilon_{1}-\epsilon_2, \epsilon_2 - \epsilon_{3} \rangle$ and the circular order is given by $-\epsilon_{1}-\epsilon_2, -\epsilon_1-\epsilon_3, \epsilon_2 - \epsilon_{3}, \epsilon_1 + \epsilon_{2}, \epsilon_1 + \epsilon_3, \epsilon_{3} - \epsilon_2$. 
    \item[$F_4^1$:]If $\ba' = \a_i$ for $i >1$, then $\Phi_{\bb, \bb'}$ is a product of rank 1 root systems. If $\ba' = \a_1$, then $\Phi_{\bb, \bb'} = \langle -\epsilon_{1}-\epsilon_2, \epsilon_2 - \epsilon_{3} \rangle$ and the circular order is given by $-\epsilon_{1}-\epsilon_2, -\epsilon_1-\epsilon_3, \epsilon_2 - \epsilon_{3}, \epsilon_1 + \epsilon_{2}, \epsilon_1 + \epsilon_3, \epsilon_{3} - \epsilon_2$. 
    \item[$G_2$:] If $\ba' = \a_1$, then $\Phi_{\bb, \bb'}$ is a product of rank 1 root systems. If $\ba' = \a_2 = -2\epsilon_1+\epsilon_2 +\epsilon_3$ and $\bb = \epsilon_1+\epsilon_2 - 2\epsilon_3$, then $\Phi_{\bb, \bb'} = \langle \bb. \bb'\rangle$ is a rank 2 root system such that $\bb+\bb' = 2\epsilon_2 - \epsilon_1 - \epsilon_3$ is a root and the circular order is given by $ \bb, \bb+\bb', \bb', -\bb, -(\bb+\bb'), -\bb'$.
    \item[$G_2^1$:] If $\ba' = \a_1$, then $\Phi_{\bb, \bb'}$ is a product of rank 1 root systems. If $\ba' = \a_2 = -2\epsilon_1+\epsilon_2 +\epsilon_3$ and $\bb = \epsilon_1+\epsilon_2 - 2\epsilon_3$, then $\Phi_{\bb, \bb'} = \langle \bb. \bb'\rangle$ is a rank 2 root system such that $\bb+\bb' = 2\epsilon_2 - \epsilon_1 - \epsilon_3$ is a root and the circular order is given by $ \bb, \bb+\bb', \bb', -\bb, -(\bb+\bb'), -\bb'$.
\end{enumerate}
\subsection{On the existence of representatives of affine simple reflections that satisfy Coxeter relations}
Recall that for a connected, reductive group $H$ over $F$, we say that $H_\bF$ is of even (resp. odd) unitary type if $H_\bF^\Sc \cong \Res_{L/\bF} \SU_n$ where $\SU_n$ is the special unitary group defined by a ramified quadratic extension $\tF/L$ and with $n$ even (resp. $n$ odd). Similarly, we say $H_\bF$ is of odd orthogonal type if $H_\bF^\Sc \cong \Res_{\tF/\bF} \mathrm{Spin}_{2n+1}$.

First, let us take up the case when $G_\bF^{\der}$ does not contain a simple factor of even unitary type. Note that by \cite[\S 4.2.23]{BT2}, this is the same as assuming that no simple factor of $G_\bF^\der$ has affine root system is of type $B-C_n$.  Also note that by \cite[\S 4.2.23]{BT2}, saying that $G_\bF^\der$ does not contain a simple factor of odd orthogonal type is the same as assuming that no simple factor of $G_\bF^\der$ has affine Dynkin diagram of type $B_n$.
\begin{corollary}\label{cor:Coxeter} Let $G$ be  connected, reductive group over $F$. Assume that $G_\bF^{\der}$ does not contain a simple factor of even unitary type. For $\ba, \ba' \in \breve\Delta$, let $n_{s_\ba}, n_{s_{\ba'}}$ be the representatives in $G_\bF$ of $s_\ba$ and $s_{\ba'}$ respectively defined as in \eqref{Rep1}, \eqref{Rep2}, and \eqref{def:affinerep}. Then the elements $n_{s_\ba}, n_{s_{\ba'}}$ satisfy Coxeter relations. 
    \end{corollary}
    \begin{proof}
        Since we are only dealing with the elements of the affine Weyl group over $\bF$, we may and do assume that $G_\bF$ is almost $\bF$-simple and simply connected. The hypothesis of the corollary then says that $G_\bF^\der$ does not have affine Dynkin diagram of type $B-C_n$. Now the corollary follows from the discussion in \S \ref{sec:CO} and Lemma \ref{lem:CR} if $G_\bF$ is not of odd orthogonal type. So we may and do assume $G_\bF$ is almost $\bF$-simple and its affine Dynkin diagram is of type $B_n$. Then note that $G_\tF$ is split and its affine Dynkin diagram is a product of affine Dynkin diagrams of type $B_n$ and the elements of $\Aut(\tF/\bF)$ permutes these factors transitively. Without loss of generality, assume $\ba \in \breve\Delta \backslash \breve\Delta_0$ and $\ba' \in \breve\Delta_0$. We refer to case $B_n$ in  Section \ref{sec:CO}. The only case where \cite{BT2}[Proposition 6.1.8] cannot be applied to verify that the representatives satisfy Coxeter relations is when $\bb  = -\epsilon_1-\epsilon_2$ and $\bb' = \epsilon_1 - \epsilon_2$.   The Coxeter relation for these reflections is $s_\ba s_{\ba'} = s_{\ba'} s_\ba$. 
        
        Let $\tilde\Delta$ be the set of affine roots of $G(\tF)$ (relative to $T$) that contains $\tilde\Delta_0$ and whose bounding hyperplanes form an alcove in the apartment $\CA(T,\tF)$. Then there exist $\ta \in \tilde\Delta \backslash \tilde\Delta_0$, $\ta' \in \tilde\Delta_0$ such that
        \begin{itemize}
        \item $\ta, \ta'$ lie in a same connected component of the affine Dynkin diagram over $\tF$,
            \item $s_\ta s_{\ta'} = s_{\ta'} s_\ta$,
            \item the gradient $\tb$ of $\ta$ lies in $ \tilde\Phi(G,T)$, $\tb|_{S} = \bb$,
            \item $\ta'|_S = \ba'$.  
        \end{itemize}
        Further, $n_{s_\ta} = \tb^\vee(\varpi_\tF) n_{s_\tb}$, $n_{s_\ba} = \prod_{\gamma \in \Aut(\tF/\bF)} \gamma(n_{s_\ta})$ and $n_{s_{\ba'}} = \prod_{\gamma \in \Aut(\tF/\bF)}\gamma( n_{s_{\ta'}})$. 
       Now, $n_{s_\ta}n_{s_{\ta'}} = n_{s_{\ta'}} n_{s_\ta}$ by \cite[Lemma 3.1]{Gan15}. Then 
       \begin{align*}
           n_{s_\ba} n_{s_{\ba'}} &= \prod_{\gamma \in \Aut(\tF/\bF)} \gamma(n_{s_\ta}n_{s_{\ta'}})\\   
           & = \prod_{\gamma \in \Aut(\tF/\bF)} \gamma(n_{s_{\ta'}} n_{s_\ta})\\
           & = n_{s_{\ba'}} n_{s_{\ba}}
       \end{align*}
       This finishes the proof of the corollary.
    \end{proof}
   \begin{remark}
       We mention a correction to the proof of \cite[Lemma 5.1]{GH23}. When $G_\bF^\der$ is $\bF$-split and has a simple factor whose affine Dynkin diagram is of type $B_n$, \cite[Proposition 6.1.8]{BT1} cannot be used.  However, we may use \cite[Lemma 3.1]{Gan15} instead, as we have done above. So the statement of \cite[Lemma 5.1]{GH23} is still correct. 
   \end{remark} 
    \subsubsection{The case when $G_\bF$ is an even unitary group} In  \cite[\S 5.1]{GH23}, we showed that the Tits group of the affine Weyl group of the unitary group $\U_6 \subset \Res_{\tF/\bF} \GL_6$ does not exist when $\tF/\bF$ is wildly ramified. Note that the existence of Tits group requires both conditions 2(a) and 2(b) of Definition \ref{def:tits2} to hold. We now relax the requirement that 2(a) holds and show that when $G_\bF =\U_{2r} \subset \Res_{\tF/\bF} \GL_{2r}$, there do not even exist representatives $\{n_{s_\ba}\;|\; \ba \in \breve\Delta\}$ that satisfy Coxeter relations, when $\tF/\bF$ is a (ramified) quadratic extension. In particular, this example, combined with the results in \cite{GH23} show that a Tits group of the affine Weyl group of an even unitary group over $F$ exists if and only if it splits over an unramified extension of $F$.

    \begin{remark}\label{Rem:Morris} The question of existence of representatives of the affine simple reflections that satisfy Coxeter relations has been considered in literature; see \cite[Proposition 5.2]{Mor93}. The proof there  also relies on \cite[Proposition 6.1.8]{BT1}. However, when considers the affine root system of a ramified even unitary group $U_{2r}, r \geq 3$, its affine Dynkin diagram is of type $B - C_n$, and as noted in \S \ref{sec:CO}, a certain finite rank 2 root system involving the gradient of the affine simple root does not satisfy the hypothesis of \cite[Proposition 6.1.8]{BT1}, so this result cannot be applied in the case of a ramified even unitary group either. It is still natural to ask if there exist representatives $\{n_{s_\ba}\;|\; \ba \in \breve\Delta\}$ that satisfy all the Coxeter relations when $G$ is a ramified even unitary group, and Proposition \ref{prop:eug} gives a negative answer to this question. 
    \end{remark}

 Let $G$ be a connected reductive group over $F$ with $G_\bF = \U_{2r}\subset \Res_{\tF/\bF}\GL_{2r}$, $r\geq 3$, where $\tF/\bF$ is a (ramified) quadratic extension. Let $\tau$ denote the generator of $\Gal(\tF/\bF)$.  Let $T_\bF$ denote the standard maximal torus in $G_\bF$ and $S_\bF$ the maximal $\bF$-split component of $T_\bF$. Then 
\begin{gather*} \breve\Phi(G,S) = \{\pm \epsilon_i \pm \epsilon_j\;|\; 1\le i<j\le r\} \cup \{\pm2 \epsilon_i\;|\; 1 \le i \le r\}, \\ 
\breve\Phi_{\af}(G,S) = \{ \pm \epsilon_i \pm \epsilon_j +\frac{1}{2}\BZ\;|\;1 \le i<j \le r\} \cup \{\pm 2\epsilon_i+\BZ\;|\; 1 \le i\le r\}. 
\end{gather*}
The hyperplanes with respect to the roots $\ba_1:= \epsilon_1-\epsilon_2, \ba_2:=\epsilon_2-\epsilon_3, \cdots, \ba_{r-1}: = \epsilon_{r-1} - \epsilon_r, \ba_r:= 2\epsilon_r, \ba_0:=-\epsilon_1 - \epsilon_2 +\frac{1}{2}$ form an alcove in $\CA(S,\bF)$ which we denote as $\bal$. Let $\breve\Delta_0 =\{\ba_i\;|\; 1\le i \le r\}$ and $\breve\Delta = \breve\Delta_0 \cup \{ \ba_0\}$. Let $\bs_i = s_{\ba_i}, 0 \le i \le r$. 

\begin{proposition}\label{prop:eug}
    With notation as in the preceding paragraph, there don't exist representatives $m(\bs_i) \in G(\bF)$ of $\bs_i, 0 \leq i \leq r$ such that $\{m(\bs_i)\;|\; 0 \leq i \leq r\}$ satisfy all the Coxeter relations.
\end{proposition}
\begin{proof} For $\sigma$ a permutation in the symmetric group $S_{2r}$, let $g_\sigma$ denote the corresponding permutation matrix in $\GL_{2r}(\tF)$ whose entries are all 0 or 1. Recall that we have chosen an extra special vertex $\bv_0 \in \CA(S,\bF)$, and a Chevalley-Steinberg system whose associated valuation is $\breve v_0$. This yields representatives $\{n_{\bs_i}\;|\; 0 \leq i \leq r\}$ as in equations \eqref{Rep1}, \eqref{Rep2}, and \eqref{def:affinerep}. With this choice of the Chevalley-Steinberg system we have, up to elements of order 2 in $T(\bF)_1$, that 
\begin{align*}
    n_{\bs_0} &= n_{\bb^\vee} g_{(1\; 2r-1)(2\; 2r)} \mod T(\bF)_1\\
    n_{\bs_i} &= g_{(i\; i+1)(2r-i\; 2r-i+1)} \mod T(\bF)_1, 1 \leq i \leq r-1,\\
    n_{\bs_r} &= g_{(r\; r+1)} \mod T(\bF)_1.
    \end{align*}
  The definition of $n_{\bb^\vee}$ is given in Section \ref{sec:rasr}; it depends on a choice $\tb^\vee \in X_*(T_{\sc})$ and a uniformizer $\varpi_\tF$ of $\tF$. In the case at hand, we have $\bb^\vee = -\epsilon_1^* - \epsilon_2^*$. We may choose $\tilde b^\vee = -\tilde \epsilon_1^* + \tilde \epsilon_{2r-1}^* \in \tilde\Phi(G,T)$, a uniformizer $\varpi_\tF$ and take $n_{\bb^\vee} = \Nm_{\tF/\bF} \tilde b^\vee(\varpi_\tF) = \diag(\varpi_\tF^{-1}, \tau(\varpi_\tF)^{-1}, 1,1, \cdots, 1, \varpi_{\tF}, \tau(\varpi_\tF))$.

    Suppose there exist representatives $m(\bs_i) \in G(\bF)$ of $\bs_i, 0 \leq i \leq r$, such that $\{m(\bs_i)\;|\; 0 \leq i \leq r\}$ that satisfy all the Coxeter relations. Then, for each $0 \leq i \leq r$,  $m(\bs_i)$ differs from $n_{\bs_i}$ by an element of $T(\bF)_1$.  For  $t_i \in T(\bF)$, write $t_i = \diag(d_{1i}, d_{2i}, \cdots, d_{ri}, \tau(d_{ri})\i, \cdots,   \tau(d_{2i})\i, \tau(d_{1i})\i)$, with $ d_{ji} \in \tF^\times$. Such a $t_i$ lies in $T(\bF)_1$ if and only if each $d_{ji} \in \fkO_{\tF}^\times$.

    Combining these observations, we see that  $m(\bs_0) = t_0 \cdot g_{(1\; 2r-1)(2\; 2r)}$ where $t_0 = (\varpi_\tF^{-1}, u\varpi_{\tF}^{-1}, \cdots,  d_{r0}, \tau(d_{r0})\i, \cdots, \tau(u^{-1}\varpi_\tF), \tau(\varpi_\tF))$ for a suitable (and possibly different) uniformizer $\varpi_\tF$ of $\tF$. 
   Also
   \begin{align*}
       m(\bs_1) &= t_1 g_{(1\; 2)(2r-1\; 2r)} &\text{ where } d_{j1} \in \fkO_\tF^\times \text{ for } 1 \leq j \leq r,\\
       m(\bs_2) &= t_2 g_{(2\; 3)(2r-2\; 2 r-1)} &\text{ where } d_{j2} \in \fkO_\tF^\times \text{ for } 1 \leq j \leq r,\\
       &\vdots\\
        m(\bs_{r-1}) &= t_{r-1} g_{(r-1\; r)(r+1\; r+2)} &\text{ where } d_{j\; r-1} \in \fkO_\tF^\times \text{ for } 1 \leq j \leq r,\\
       m(\bs_r) &= t_r g_{(r\; r+1)} & \text{ where } d_{jr} \in \fkO_\tF^\times \text{ for } 1 \leq j \leq r.\\
   \end{align*}
The Coxeter relation $m(\bs_0) m(\bs_1) = m(\bs_1) m(\bs_0)$ implies that
   \begin{align}\label{(1)}
       u\tau(u) = 1 \text{ and } \tau(d_{21})\i = u d_{11}.
   \end{align}
      %The Coxeter relation $n_{\bs_0} n_{\bs_3} = n_{\bs_3} n_{\bs_0}$ implies that
   % \begin{align}\label{(2)}
   %    d_{33}\tau(d_{33}) =1.
   % \end{align}
      The Coxeter relation $m(\bs_0) m(\bs_2) m(\bs_0) = m(\bs_2) m(\bs_0) m(\bs_2)$ implies that
   \begin{align}\label{(2)}
       d_{22}d_{32} = u \varpi_\tF^{-1} \tau(\varpi_\tF) \text{ and } d_{12} \tau( d_{30})=1.
   \end{align}
      The Coxeter relation $m(\bs_1) m(\bs_2) m(\bs_1) = m(\bs_2) m(\bs_1) m(\bs_2)$ implies that
   \begin{align}\label{(3)}
      d_{31} = d_{12} \text{ and } d_{11}d_{21}d_{31}^{-1} = d_{12}^{-1}d_{22}d_{32}.
   \end{align}

   Hence    
   \[ \tau(u)^{-1} d_{11}\tau(d_{11})\i = d_{11}d_{21} = d_{22}d_{32} = u \varpi_\tF \i\tau(\varpi_\tF)\]
   Since $\tau(u) u = 1$, this implies that 
   \[d_{11}\tau(d_{11})\i = \varpi_\tF \i \tau(\varpi_\tF),\]
which implies that 
\[\tau(d_{11} \varpi_\tF) = d_{11} \varpi_{\tF}.\]
Since $d_{11}$ is a unit in $\tF$, $d_{11} \varpi_\tF$ is a uniformizer of $\tF$ and hence cannot be fixed by $\tau$.   This gives a contradiction and proves that there do not exist representatives $m(\bs_i), 0 \leq i \leq r$, that satisfy all the Coxeter relations.   
\end{proof}

\smallskip
    \subsection{On the squares of the representatives}\label{sec:squares}
    In this section, we calculate $n_{s_\ba}^2, \ba \in \breve\Delta$. Given \eqref{eq:rfsr2}, we only need to calculate $n_{s_\ba}^2, \ba \in \breve\Delta\backslash \breve\Delta_0$. Let $\bb, \bb_*, n_{s_\ba}$ be as in \S \ref{sec:rasr}. We have
    \[n_{s_\ba}^2 = n_{\bb^\vee} s_\bb(n_{\bb^\vee}) \cdot  n_{s_{\bb}}^2. \]
    Since $ n_{s_\bb}^2$ is calculated as in \eqref{eq:rfsr2}, it remains to calculate $n_{\bb^\vee} s_\bb(n_{\bb^\vee})$.

There are two possibilities.
\begin{enumerate}
    \item Suppose $ \bb_* =\bb$, that is $2 \bb_*$ is not a root. We fix a pinning $(\bF_\tb, x_\bb)$ as in \cite[\S 4.1.5 and \S 4.1.8]{BT2}. Here $\bF_\tb \hookrightarrow \tF$ and $x_\bb: {\bf U}_\bb \xrightarrow{\simeq} \Res_{\bF_\tb/\bF}\BG_a$ is an isomorphism. Then $s_\bb = \prod \gamma(s_{\tb})$ where $\tb \in X^*(T)$ is as in \S \ref{sec:rasr}, and $\gamma$ runs through the set of $\bF$-embeddings $\bF_\tb \hookrightarrow \tF$. Note that $\tb|_S = \bb$. 
    Then
    \begin{align*}
        s_\bb(n_{\bb^\vee}) &= s_\bb(n_{\bb^\vee}) = \prod\gamma(s_\tb)\left(\prod \gamma'(n_{\tb^\vee})\right)\\
    \end{align*}
    It is easy to see that for $\bF$-embeddings $\gamma, \gamma'$ of $\bF_\tb \hookrightarrow \tF$,
    \[\gamma(s_\tb)(\gamma'(n_{\tb^\vee})) = \begin{cases}
        \gamma'(n_{\tb^\vee}) & if \gamma \neq \gamma'\\
        \gamma(n_{\tb^\vee}^{-1}) & if \gamma = \gamma'.
    \end{cases}\]
    Hence
    \[s_\bb(n_{\bb^\vee})  = \prod \gamma(n_{\tb^\vee}^{-1}) = n_{\bb^\vee}^{-1}.\]
    This implies that $n_{\bb^\vee} s_\bb(n_{\bb^\vee}) = n_{\bb^\vee} n_{\bb^\vee}^{-1} = 1$, so $n_{s_\ba}^2 = n_{s_\bb}^2 = \bb^\vee(-1)$  as in \eqref{eq:rfsr2}.
    
    \item Suppose $ \bb_* \neq \bb$, that is $2 \bb_*$ is a root. Recall that we have fixed $(\bF_{\tb_*}, \bF_{\tb_*+\tb_*'}, x_{\bb_*})$ as in Section \ref{Sec:Affpin} (see\cite[\S 4.1.9]{BT2}) where $\tb_*, \tb_*' \in \tilde\Phi(G,T)$ are such that $\tb_*, \tb_*'|_S = \bb_*$ and $\tb_* +\tb_*'$ is a root. Then $\pr(\tb_*^\vee) = \bb^\vee$. Let  $\varpi_\tF$ be a uniformizer of $\tF$. Then  $n_{\bb^\vee} = \Nm_{\tF/\bF} \tb_*^\vee(\varpi_\tF) =  \tb_*^\vee(\varpi_{\bF_{\tb_*}})$ where $\varpi_{\bF_{\tb_*}} = \Nm_{\tF/\bF_{\tb_*}} \varpi_\tF$. 
     In this case
    \[s_\bb = \prod\gamma(s_{\tb_*}\; s_{\tb_*'}\; s_{\tb_*})\]
    where $\gamma$ runs over the set of $\bF$-embeddings of $\bF_{\tb_*+\tb_*'}$ in $\tF$. 
    It is easy to calculate $s_\bb(n_{\bb^\vee})$ (in fact, we use the morphism $x_{\bb_*}$ and simply perform this calculation inside $\Res_{\bF_{\tb_*+\tb_*'}/\bF} \SU_3$, where $\SU_3$ is defined by the quadratic extension $\bF_{\tb_*}/\bF_{\tb_*+\tb_*'}$), and we see that 
    \[s_\bb(n_{\bb^\vee}) = (\Nm_{\bF_{\tb_*}/\bF} \tb_*^\vee(u)) \cdot n_{\bb^\vee}^{-1}\]
    where $u = \gamma_0(\varpi_{\bF_{\tb_*}})\varpi_{\bF_{\tb_*}}^{-1}$. Hence 
    \[n_{\bb^\vee} s_\bb(n_{\bb^\vee}) = \Nm_{\bF_{\bb_*}/\bF} \tb_*^\vee(u)\]
    so
    \[n_{s_\ba}^2 = \Nm_{\bF_{\bb_*}/\bF} \tb_*^\vee(u) \cdot n_{s_{\bb}}^2 = \Nm_{\bF_{\bb_*}/\bF} \tb_*^\vee(u)\]
    where the last equality above follows from \eqref{eq:rfsr2}. 

    \end{enumerate}
    \begin{remark}\label{rem:tameodd}
        If the extension $\bF_{\tb_*}/\bF_{\tb_* +\tb_*'}$ is tamely ramified and if the uniformizer $\varpi_{\bF_{\tb_*}}$ satisfies $\gamma_0(\varpi_{\bF_{\tb_*}})= -\varpi_{\bF_{\tb_*}}$, then $u=-1$ and $n_{s_\ba}^2 = \Nm_{\bF_{\bb_*}/\bF} \tb_*^\vee(-1) = \bb^\vee(-1)$ is an element of order 2 and lies in $\breve S_2$. 
    \end{remark}
    \begin{theorem}\label{thm:TitsbF}
Suppose $G$ is a connected, reductive group over $F$ such that $G_\bF^\der$ does not contain a simple factor of unitary type. There exist a set of representatives $\{n_{s_\ba}\;|\; \ba \in \breve\Delta\}$ that
\begin{enumerate}
    \item satisfy Coxeter relations,
    \item $n_{s_\ba}^2$ is an element of order 2 and lies in $\breve S_2$.
\end{enumerate} 
Let $\breve\CT_\af$ be the subgroup of $G(\bF)$ generated by $n_{s_\ba}, \ba \in \breve\Delta$. Then $\breve\CT_\af$ contains $\breve S_2$ and is a Tits group of $\bW_\af$ over $\bF$.
\end{theorem}
\begin{proof}
    The existence of a set of representatives $\{n_{s_\ba}\;|\; \ba \in \breve\Delta\}$ that satisfies (1) and (2) follows from Corollary \ref{cor:Coxeter} and \S \ref{sec:squares}. The proof that $\breve\CT_\af$ is a Tits group of $\bW_\af$ over $\bF$ is identical to \cite[Propsition 5.3]{GH23}. 
\end{proof}
\begin{remark}\label{rem:TitsOdd}
    Let $G = U_{2r+1}$ be the odd unitary group associated with a ramified quadratic extension $\tF/\bF$. In this case, we know that the set of representatives $\{n_{s_\ba}\; |\; \ba \in \breve\Delta\}$ of the affine simple reflections constructed in \eqref{Rep1}, \eqref{Rep2'}, and \eqref{def:affinerep} satisfy Coxeter relations. 
    
    Supose $\tF/\bF$ is tamely ramified. Since it is then always possible to choose a uniformizer $\varpi_L$ of $L$ such that $\gamma_0(\varpi_\tF) = -\varpi_\tF$, we have by Remark \ref{rem:tameodd} that $n_{s_\ba}^2 \in \breve S_2$ for all $\ba \in \breve\Delta$, so a Tits group $\breve\CT_\af$ of $\bW_\af$ exists in this case. 
    
    Now, suppose $\tF/\bF$ is wildly ramified. In this case, there does not exist a uniformizer $\varpi_\tF$ of $\tF$ such that $\gamma_0(\varpi_\tF) = -\varpi_\tF$. Hence the representative $n_{s_\ba}$ of the affine simple reflection $s_\ba$ constructed in \eqref{def:affinerep} is not generally of order 4. One could still ask if there exist representatives $\{m(s_\ba) \in G(\bF)\;|\; \ba \in \breve\Delta\}$, perhaps arising out of a different construction, such that the group generated by these $m({s_\ba})$'s yields a Tits group of $\bW_\af$. Note that if $m(s_\ba)$ is a representative in $G(\bF)$  of $s_\ba$ then $m(s_\ba) = t_\ba n_{s_\ba}$ for some $t_\ba \in T(\bF)_1$. Consider  $\ba \in \breve\Delta \backslash \breve\Delta_0$. Then $m(s_\ba)^2 = t_\ba s_\ba(t_\ba) n_{s_\ba}^2$. Write $t_\ba = \diag(u_1, u_2, \cdots u_n, u, \gamma_0(u_n)^{-1}, \cdots \gamma_0(u_1)^{-1})$ for suitable $u_i, u \in \fkO_\tF^\times$. If $m(s_\ba)^2$ has order 2,  the calculation in case (2) of \S\ref{sec:squares} implies that $u_1^{-1} \gamma_0(u_1)= \pm \gamma_0(\varpi_\tF)\varpi_\tF^{-1}$, but this means that $\gamma_0(u_1^{-1}\varpi_\tF) = \pm u_1^{-1}\varpi_\tF$. Since $u_1^{-1} \varpi_\tF$ is a uniformizer of $\tF$, this condition cannot be satisfied when $\tF/\bF$ is wildly ramified.  This implies that when $\tF/\bF$ is wildly ramified, $\breve\CT_\af$ does not exist. 
\end{remark}

\section{Tits groups of affine Weyl groups over $F$}\label{sec:descend}
The goal of this section is to descend the construction of the Tits group of the affine Weyl group over $\bF$ down to $F$. 
\subsubsection{} \label{sec:IndTor}
 Let $T$ be an induced torus over $F$. Then there exist representatives $\{n_{\breve \lambda}\;|\; {\breve \lambda} \in X_*(T)_{\Gamma_0}\}$ that forms a group and is $\sigma$-stable. To see this, we may reduce ourselves to the case where $T = \Res_{L/F} \BG_m$ with $L/F$ a finite separable extension of degree $n$. Let $\tL$ be the Galois closure of $L$ in $F_s$. The cocharacter lattice $X_*(T)$ admits a $\BZ$-basis $\CB: = \{\tilde\lambda_1, \cdots , \tilde\lambda_n\}$ that is permuted transitively by $\Gamma$. We may and do assume that the fixed field of $\tilde\lambda_1$ is $L$. Fix a uniformizer $\varpi_L$ of $L$. 
Let $\tF = \tL \bF$.  Choose a uniformizer $\varpi_\tF$ such that $\Nm_{\tF/L\bF} \varpi_\tF = \varpi_L$ (see \cite[ Chapter V, \S 4,  Proposition 7]{Ser79}). Let ${\breve \lambda}_1 = \pr(\tilde\lambda_1)$. Then with $f = [L\bF:\bF]$, we have that $\{\sigma^{i-1}({\breve \lambda}_1)\;|\; 1 \leq i \leq f\}$ is a $\BZ$-basis of $X_*(T)_{\Gamma_0}$. Define $n_{{\breve \lambda}_1} = \Nm_{\tF/\bF} \tilde\lambda_1(\varpi_\tF) = \Nm_{L\bF/\bF} \tilde\lambda_1(\varpi_L)$. We claim that $\sigma^f$ fixes $n_{{\breve \lambda}_1}$. Let $\tilde\sigma$ be any lift of $\sigma$ to $\Gamma$. Since $\sigma^f$ fixes ${\breve \lambda}_1$ ( in fact, $\sigma^f$ acts as identity on $X_*(T)_{\Gamma_0}$), we have that $\tilde\sigma^f(\tilde\lambda_1) = \gamma(\tilde\lambda_1)$ for a suitable $\gamma \in \Gamma_0$. Then $\gamma^{-1} \tilde\sigma^f$ fixes $\tilde\lambda_1$, so $\gamma^{-1} \tilde\sigma^f \in \Gal(F_s/L)$. Hence $\sigma^f(n_{{\breve \lambda}_1}) = \Nm_{\tF/\bF} \tilde\sigma^f(\tilde\lambda_1(\varpi_\tF)) = \Nm_{L\bF/\bF} \tilde\lambda_1(\gamma^{-1}\tilde\sigma^f(\varpi_L)) = \Nm_{L\bF/\bF} \tilde\lambda_1(\varpi_L) =n_{{\breve \lambda}_1}$. Now, suppose ${\breve \lambda} \in X_*(T)_{\Gamma_0}$. Then ${\breve \lambda}$ can be uniquely written as ${\breve \lambda} = \sum_i c_i \sigma^{i-1}({\breve \lambda}_1)$. Define $n_{{\breve \lambda}} = \prod_i \sigma^{i-1}(n_{{\breve \lambda}_1})^{c_i}$. The set $\{n_{{\breve \lambda}}\;| {\breve \lambda} \in X_*(T)_{\Gamma_0}\}$ clearly forms a group. The fact that this set of representatives is $\sigma$-stable follows from the fact that $\sigma^f(n_{{\breve \lambda}_1}) = n_{{\breve \lambda}_1}$. 

Suppose $G$ is quasi-split over $F$. Then the torus $T_\Sc$ is induced (see \cite[Section 4.4.16]{BT2}). The above discussion then says that we may choose $\{n_{\bb^\vee} \in T_\Sc(\bF)\;|\; \bb^\vee \in X_*(T_{\Sc})_{\Gamma_0}\}$ that forms a group and is $\sigma$-stable. This observation will be used in the proof of Proposition \ref{prop:descent}.

We now come to the main result of this section.
   \begin{proposition}\label{prop:descent}
Suppose $G$ is a connected, reductive group over $F$ such that $G_\bF^\der$ does not contain a simple factor of even unitary type. There exists a set of representatives $\{n_{s_\ba}\;|\; \ba \in \breve\Delta\}$ that
\begin{enumerate}
\item is $\sigma$-stable,
    \item satisfy all the Coxeter relations.
\end{enumerate} 
\end{proposition}
\begin{proof}
Let $\CX$ be a $\s$-orbit in $\breve\Delta$ such that $\bW_\CX$ is finite. Fix $\ba \in \CX$ and let $\bb$ be the gradient of $\ba$. Since $W_{\CX}$ is finite, we have $\bw_{\CX} \in \bW_{\CX}^\s$. In particular, $W_\CX^\sigma \neq 1$. Thus $\bb|_A \neq 0$ and hence is a root $b$ in $\Phi(G,A)$. 
We have two possibilities.\\\\
\noindent (I) Suppose $\bb_*  = \bb$, that is $2\bb_*$ is not a root. The argument in this case is a simple modification of \cite[Proposition 6.1]{GH23}. We write the details for completeness. Let $m_{\varpi_{\bF_\tb}}: \BG_a \rightarrow  \BG_a$ denote $\bF_\tb$-morphism given by multiplication by $\varpi_{\bF_\tb}$, where $\varpi_{\bF_\tb}$ is as in Section \ref{sec:rasr}. Define $x_{\ba}:= x_{\bb} \circ\Res_{\bF_\tb/\bF} m_{\varpi_{\bF_\tb}}$ for $\ba \in \breve\Delta \backslash \breve\Delta_0$, and $x_\ba = x_\bb$ for $\ba \in \breve\Delta_0$. Note that $x_{\ba}$ is a $\bF$-isomorphism from $\Res_{\bF_\tb/\bF} \BG_a$ to $U_{\bb}$.  Let $k = \# \CX$. We show that 

(a) there exists $u \in \fkO_{\bF_\bb}^\times$ such that $\sigma^k(x_\ba(u)) = x_\ba(u)$. 

Let $v \in \CA(A,F) \subset \CA(S, \bF)$ and $r \in \BR$. For $\bb \in \breve\Phi(G,S)$, let $U_\bb(\bF)_{v,r} \subset U_\bb(\bF)$ denote the filtration of root subgroup $U_{\bb}(\bF)$ as in \cite[\S 4.3]{BT2}. We recall the definition of the filtration of the root subgroup $U_b(F)$ (cf. \cite[\S 5.1.16 - 5.1.18]{BT2}). Let $\breve\Phi^b: = \{ \bc \in \breve\Phi(G,S)\;|\; \bc|_A = b \text{ or } 2b\}$. This is a $\s$-stable positively closed subset of $\breve\Phi(G,S)$; that is if $\bc_1, \bc_2 \in \breve\Phi^b$ such that $\bc+\bc'$ is a root, then $\bc+\bc' \in \breve\Phi^b$. For any fixed ordering, the subset
\begin{align}\label{ubr}
    U_b(\bF)_{v,r} := \prod_{\bc \in \breve\Phi^b, \bc|_A = b} U_\bc(\bF)_{v,r}  \prod_{\bc \in \breve\Phi^b_{nd}, \bc|_A = 2b} U_\bc(\bF)_{v,2r}
\end{align}
is a subgroup of $U_b(\bF)$. Let $U_b(F)_{v,r} := U_b(\bF)_{v,r} \cap U_b(F)$. 

We have
\begin{align}\label{ubvv0}
    U_\bb(\bF)_{v,r} = U_{\bb}(\bF)_{\bv_0, r- \bb(v - \bv_0)},
\end{align}
where $\bv_0$ is the special vertex in $\CA(S,\bF)$ fixed in \S \ref{notation}. 

The pinning $x_\ba: \Res_{\bF_\tb/\bF} \BG_a \rightarrow U_{\bb}(\bF)$ satisfies $x_{\ba}(\fkO_{\bF_\tb}) = U_\bb(\bF)_{\bv_0,r_0}$ and  $x_{\ba}(\fkp_{\bF_\tb}) = U_{\bb}(\bF)_{\bv_0,s_0}$ for a suitable $r_0<s_0$. 
%Now, let $r = r_0 - \bb(v - \bv_0)$ and $s = s_0 - \bb(v - \bv_0)$. 
Using \eqref{ubvv0},
% \[U_{\bb}(\bF)_{ v, r} = U_{\bb}(\bF)_{ \bv_0, r_0}\; \text{ and  }\; U_{\bb}(\bF)_{ v, s} = U_{\bb}(\bF)_{ \bv_0, s_0}.\]
and then adjusting $r$ if necessary, we ensure 
\begin{align}\label{xabf}
x_{\ba}(\fkO_{\bF_\tb}) = U_{\bb}(\bF)_{v,r},\;\; x_{\ba}(\fkp_{\bF_\tb}) = U_{\bb}(\bF)_{ v, r+}
\end{align}
for a suitable $r \in \BR$. In particular, $U_{\bb}(\bF)_{v,r}\neq U_{\bb}(\bF)_{v,r+}$.
Now \cite[Proposition 5.1.19]{BT2} implies  that $U_b(F)_{v,r} \neq U_b(F)_{v,r+}$. In other words, there exists $u' \in U_b(F)_{v,r}$ such that
\[ u' \notin \prod_{\bc \in \breve\Phi^b, \bc|_A = b} U_\bc(\bF)_{v,r+}  \prod_{\bc \in \breve\Phi^b_{nd}, \bc|_A = 2b} U_\bc(\bF)_{v,2r}.\]
Note that $\sigma^k$ fixes every element of $\CX$ (and hence also every element of $\Phi^b$) and $\sigma^k(u') = u'$. By \eqref{ubr}, we have
\[u' \in \prod_{\bc \in \breve\Phi^b, \bc|_A = b} U_\bc(\bF)_{v,r}^{\sigma^k}  \prod_{\bc \in \breve\Phi^b_{nd}, \bc|_A = 2b} U_\bc(\bF)_{v,2r}^{\sigma^k},\]
and 
\[u' \notin \prod_{\bc \in \breve\Phi^b, \bc|_A = b} U_\bc(\bF)_{v,r+}^{\sigma^k} \prod_{\bc \in \breve\Phi^b_{nd}, \bc|_A = 2b} U_\bc(\bF)_{v,2r}^{\sigma^k}.\]

Thus there exists $\bc \in \Phi^b, \bc|_A = b$ such that
$ U_\bc(\bF)_{v,r+}^{\sigma^k}  \subsetneq U_\bc(\bF)_{v,r}^{\sigma^k}$. Since $\bc = \sigma^i(\bb)$ for a suitable $i<k$, we also have
\[U_\bb(\bF)_{v,r+}^{\sigma^k}  \subsetneq U_\bb(\bF)_{v,r}^{\sigma^k}.\]

Let $u'' \in U_\bb(\bF)_{v,r}^{\sigma^k}\backslash U_\bb(\bF)_{v,r+}^{\sigma^k}$ and $u= x_\ba^{-1}(u'')$. Then $\sigma^k(x_\ba(u)) = x_\ba(u)$. By \eqref{xabf}, $u\in \fkO_{\bF_\bb}^\times$.

(a) is proved. 

Since $x_\ba(u) x_{-\ba}(u^{-1}) x_\ba(u) \in N_G(S)(\bF)$, we have $$x_\ba(u) \sigma^k(x_{-\ba}(u^{-1}))x_\ba(u)=\s^k(x_\ba(u)) \sigma^k(x_{-\ba}(u^{-1})) \s^k(x_\ba(u))\in N_G(S)(\bF).$$ The uniqueness assertion in \cite[\S 6.1.2, (2)] {BT2} implies that $\sigma^k(x_{-\ba}(u^{-1})) = x_{-\ba}(u^{-1})$.

Let $x_\ba' = x_\ba \circ m_u$, where $m_u$ is the multiplication by $u$. We consider the pinning $\{x_\ba', \sigma \circ x_\ba', \cdots \sigma^{k-1} \circ x_{\ba}'\}$. Then $x_{\ba}'(1) =x_{\ba}(u)$ and $x_{-\ba}'(1) = x_{-\ba}(u^{-1})$. For $\bc \in \CX$, let $n_{s_\bc}' = x_{\bc}'(1)x_{-\bc}'(1)x_{\bc}'(1)$ be the representative in $N_G(S)(\bF)$ of $s_\bc$ obtained using this pinning. Then $\sigma ^{k}(n_{s_\bc}' )= n_{s_\bc}'$ and the set $\{n_{s_\bc}'\;|\; \bc \in \CX\} = \{n_{s_\ba}', \sigma(n_{s_\ba}'), \cdots, \sigma^{k-1}(n_{s_\ba}')\}$ is $\sigma$-stable.\\

\noindent (II) Suppose $\bb_* \neq \bb$, that is $2\bb_*$ is a root. We need to show that there exists a set of representatives $\{n_{s_\ba}\;|\; \ba \in \CX\}$ that is $\s$-stable. Without loss of generality, we may assume that $G$ is almost $F$-simple and $G_\bF$ is almost $\bF$-semisimple. Write $G_{\bF,\ad} = \prod_i G_{\bF, \ad}^i$. Since $\sigma$ leaves $G_{\bF,\ad}$ stable, it permutes the factors $G_{\bF, \ad}^i$. Since $G$ is almost $F$-simple, $\sigma$ indeed permutes these factors transitively. Now, since $\bb_*$ is  multipliable, we see that one factor of $G_{\bF,\ad}$, and hence every factor of $G_{\bF, \ad}$ is, up to restriction of scalars, an odd unitary group of adjoint type. In this case, it is well-known that $G$ is quasi-split over $F$. 
It may be proved as follows. Note that the $F$-inner forms of $G$ are parametrized by the pointed cohomology set $H^1(\Gamma, G_{\ad}(F_s)).$ If we show that $H^1(\Gamma, G_{\ad}(F_s))=1$, then since any connected, reductive group over $F$ admits a unique quasi-split inner form, we see that $G$ is itself quasi-split over $F$. Using \cite[Proposition 13.1 (1)]{Kot16} and the fact that $G_{\ad}$ has trivial center, we have a canonical bijection \[ \kappa: H^1(\Gamma, G_{\ad}(F_s)) \rightarrow \left(X_*(T_{\ad})/X_*(T_{\Sc})\right)_{\Gamma}.\]
 Since $ \left(X_*(T_{\ad})/X_*(T_{\Sc})\right)_{\Gamma}\cong \left(\left(X_*(T_{\ad})/X_*(T_{\Sc})\right)_{\Gamma_0}\right)_{\Gamma/\Gamma_0}$ 
it would suffice to prove that
\begin{align}\label{H1trivial}
    \left(X_*(T_{\ad})/X_*(T_{\Sc})\right)_{\Gamma_0}=1.
\end{align}

Let $n$ odd, let $H = \U_{L/E}(n)_{\ad}$, and let $\tH = Res_{E/\bF}H$, where $L/\bF$ is finite separable, and $L/E$ is quadratic. Let $S^H$ and $T^H$ (resp. $S^\tH$ and $T^\tH$) denote the maximal $E$-split (resp. $\bF$-split) torus and maximal torus of $H$ and $\tH$ respectively. To prove \eqref{H1trivial} it would suffice to prove that $ \left(X_*(T^\tH)/X_*(T_{\Sc}^\tH)\right)_{\Gamma_0}=1$ and this would amount to proving that $ \left(X_*(T^H)/X_*(T_{\Sc}^H)\right)_{\Gal(L/E)}=1.$ 
Since $H_{L} = \PGL_n$, we have $X_*(T^H)/X_*(T_{\Sc}^H) \cong \BZ/n\BZ.$ Making the action of the non-trivial element $\gamma$ of $\Gal(L/E)$ on a cyclic generator $\rho$ of $X_*(T^H)/X_*(T_{\Sc}^H)$ explicit, we see that $\tau(\rho) = \rho^{-1}$. Hence
\[\left(X_*(T^H)/X_*(T_{\Sc}^H)\right)_{\Gal(L/E)} =\langle \rho\rangle/\langle \rho^2 \rangle =1\]since $n$ is odd. Consequently we have that $G$ and $G_\bF$ are both quasi-split, $G_\tF$ is split, and $G(F) = G(\bF)^\sigma$. Hence, we may and do assume that the Chevalley-Steinberg system  $x_\tc: \BG_a \rightarrow U_\tc$, for $\tc \in \tilde\Phi(G,T)$, has the property that
$\gamma \circ x_\tc \circ \gamma^{-1} = x_{\gamma \cdot \tc}$ for any $\gamma \in \Gamma$. Let $n_{s_{\tc}}$ be the elements defined using this pinning as in \eqref{tilderep}. Note that \begin{align}\label{eq:GammaSteq}
    \gamma(n_{s_{\tc}}) = n_{s_{\gamma (\tc)}}
\end{align} 
for all $\gamma \in \Gamma$ and for all $\tc \in \tilde\Phi(G,T)$ whose restriction to $A$ is non-divisible.
Note that since $G$ is quasi-split over $F$, $\sigma$ stabilizes $\breve\Delta_0$. Let $\{n_{s_\ba}\;|\; \ba \in \breve\Delta_0\}$ be the representatives  in \eqref{Rep1} and  \eqref{Rep2}. It follows from \eqref{eq:GammaSteq} that $\gamma(n_{s_{\ba}}) = n_{s_{\ba}}$ for any  $\gamma \in \Gamma_0$ and hence  $\sigma(n_{s_{\ba}}) = n_{s_{\sigma(\ba)}}$ for all $\ba \in \breve\Delta_0$. 
It remains to choose $\{n_{s_\ba}\;|\; \ba \in \breve\Delta \backslash \breve\Delta_0\}$ such that $\sigma(n_{s_\ba}) = n_{s_{\sigma(\ba)}}$ for all $\ba \in \breve \Delta \backslash \breve\Delta_0$. We write $n_{s_\ba} = n_{\bb^\vee} n_{s_{\bb}}$ (see Section \ref{sec:rasr}) where $n_{s_\bb}$ is chosen using the Chevalley-Steinberg system above and where $n_{\bb^\vee}$ is as in Section \ref{sec:IndTor}. Note that since $G$ is quasi-split over $F$, the torus $T_\Sc$ is induced, so the discussion in Section \ref{sec:IndTor} applies. It follows from the discussion there that $n_{\bb^\vee} = \Nm_{\tF/\bF} \tilde\lambda(\varpi_\tF)$  for some $\tilde\lambda \in X_*(T_\Sc)$ such that $\pr(\tilde\lambda) = \bb^\vee$ and for some uniformizer $\varpi_\tF$ of $\tF$. Further,
\begin{align}\label{sstablebbd}
    \sigma(n_{\bb^\vee}) = n_{\sigma(\bb^\vee)}.
\end{align} 
Also using equation \eqref{eq:GammaSteq}, we have 
\begin{align}\label{sstablerep}
    \sigma(n_{s_\bb}) = n_{\sigma(s_\bb)}
\end{align}
Combining \eqref{sstablebbd} and \eqref{sstablerep}, we have $\sigma(n_{s_{\ba}}) = n_{s_{\sigma(\ba)}}$ for all $\ba \in \breve\Delta$. So the set $\{n_{s_\ba}\;|\; \ba \in \breve\Delta\}$ is $\sigma$-stable.

This proves (1). The claim that this set of representatives satisfy Coxeter relations follows from Corollary \ref{cor:Coxeter}.
\end{proof}

\begin{corollary}\label{cor:TitsF}
    Let $G$ be a connected, reductive group over $F$ such that $G_\bF^\der$ does not contain a simple factor of unitary type. Let $\{n_{s_\ba} \;|\; \ba \in \breve\Delta\}$ be a $\sigma$-stable set of representatives of the affine simple reflections in $\breve\Delta$ as in Proposition \ref{prop:descent}. Let $ \breve\CT_\af$ be the corresponding affine Tits group over $\bF$ generated by these $n_{s_\ba}\; \ba \in \breve\Delta$. Then $\CT_\af :=\breve\CT_\af^\sigma$ is a Tits group of $W_\af$ over $F$.
    \end{corollary}
\begin{proof}
   This proof is identical to the proof of \cite[Theorem 6.4]{GH23}. So we will just explain the construction of the Tits cross-section and refer to \cite[\S 6.7]{GH23} for the remaining details.   For any $a  \in \Delta$, we have $s_a=\bw_{\CX}$ for some $\s$-orbit $\CX$ in $\breve \Delta$ with $\bW_{\CX}$ finite (see \S \ref{sec:WF}). Let $\bw_{\CX}=\bs_{i_1} \cdots \bs_{i_n}$ be a reduced expression of $\bw_{\CX}$ in $\bW_{\af}$. Then $\bw_{\CX}=\s(\bs_{i_1}) \cdots \s(\bs_{i_n})$ is again a reduced expression of $\bw_{\CX}$ in $\bW_{\af}$. We have \begin{align*} n_{\bw_{\CX}} &=n_{\bs_{i_1}} \cdots n_{\bs_{i_n}}=n_{\s(\bs_{i_1})} \cdots n_{\s(\bs_{i_n})}=\s(n_{\bs_{i_1}}) \cdots \s(n_{\bs_{i_n}}) \\ &=\s(n_{\bw_{\CX}}).
\end{align*}

In particular, $n_{s_\ba}=n_{\bw_{\CX}} \in \CT_\af=\breve \CT_\af^{\s}$. 

Let $w \in W_{\af}$ and $s_{i_1} \cdots s_{i_n}$ be a reduced expression of $w$ in $W_\af$. We set $n_w=n_{s_{i_1}} \cdots n_{s_{i_n}}$. Then $n_w \in \CT$. Suppose that $s'_{i_1} \cdots s'_{i_n}$ is another reduced expression of $w$ in $W$. By \S \ref{sec:WF} (a), $\breve \ell(w)=\breve \ell(s_{i_1})+\cdots+\breve \ell(s_{i_n})=\breve \ell(s'_{i_1})+\cdots+\breve \ell(s'_{i_n})$. Since $\{n_{s_\ba}\;|\; \ba\in \breve\Delta\}$ satisfies the Coxeter relations, by condition (2)(b) $^\dagger$ in \S \ref{sec:tits}, $n_{s_{i_1}} \cdots n_{s_{i_1}} =n_{s_{i_1}'}  \cdots n_{s_{i_n}'}$. In other words, $n_w$ is independent of the choice of reduced expression in $W$.  In other words, the map $\phi: \CT_\af \to W_\af$ is surjective. We have $$\ker(\phi)=\ker(\breve \phi) \cap \CT_{\af}=\breve S_2 \cap \CT_{\af}=S_2.$$ 

For the proof of the fact that $\CT_\af$ is a Tits group of $W_\af$ and $\{n_w\;|\; w \in W_\af\}$ is a Tits cross-section of $W_\af$ in $\CT_\af$, see \cite[\S 6.7]{GH23}. 
\end{proof}
\subsection{Some remarks on Tits groups of inner forms of ramified unitary groups}\label{rem:TitsIFU} Let $G = U_m \subset \Res_{L/F} \GL_m$ be the quasi-split unitary group over $F$ with $L/F$ a ramified quadratic extension. Let $G^*$ be an $F$-inner form of $G$. Let $\sigma^*$ denote the Frobenius action on $G_\bF$ so that $G^*(F) = G(\bF)^{\sigma^*}$. Let $\bW_\af$ be the affine weyl group of $G(\bF)$  and let $W_\af^* = \bW_\af^\sigma$ be the affine Weyl group of $G^*(F)$.  In this section, we remark on the existence of a Tits group $\CT_\af^*$ of $W_\af^*$. 
\subsubsection{When $m$ is odd}\label{rem:Titsodd1}
As noted in the proof of Proposition \ref{prop:descent}, $G^*$ is isomorphic to $G$ over $F$. Then Remark \ref{rem:TitsOdd} completely answers the question on the existence of Tits groups of affine Weyl groups in this case.
\subsubsection{When $m$ is even}\label{rem:Titseven} Write $m = 2r$ with $r \geq 3$ and let $G = \U_{2r} \subset \Res_{L/F} \GL_{2r}$ be the quasi-split even unitary group over $F$ with $L/F$ a ramified quadratic extension. Then $G$ has a  unique non-quasi-split inner form $G^*$. Let $W_\af^*$ be the affine Weyl group of $G^*(F)$.   Note that the affine Dynkin diagram of $G_\bF$ is of type $B-C_r$:
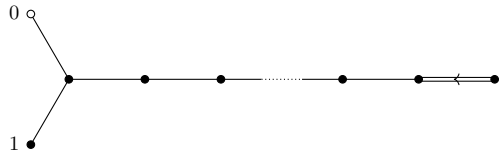
\[\dynkin[labels={0,1}, edge length=1cm]A[2]{odd}\]
The vertex labelled 0 represents the affine simple root $\ba_0$ and the vertex labelled  1 represents the finite simple root $\ba_1$. Then $\sigma^*$ acts on this affine Dynkin diagram permutes the vertices labelled 0 and 1 and leaves all the other vertices fixed. So $s_{a_0} = s_{\ba_0} s_{\ba_1}$ (see \S \ref{sec:WF}). Then the proof in  Case (1) of Proposition \ref{prop:descent} still gives representatives $\{n_{s_\ba}\;|\; \ba \in \breve\Delta\}$ that are $\sigma$-stable, but they \textit{do not} satisfy all the Coxeter relations over $\bF$. As in the proof of Corollary \ref{cor:TitsF}, this yields a set of representatives $\{n_{s_a} \in G^*(F)\;|\; a \in \Delta\}$. Now, $\{n_{s_a} \in G^*(F)\;|\; a \in \Delta\}$ satisfies all the Coxeter relations; this can be seen by a direct calculation or using \cite[Proposition 6.1.8]{BT1}. The fact that $\{n_{s_\ba}\;|\; \ba \in \breve\Delta\}$ are $\sigma$-stable, implies that 
\[n_{s_{\ba_0}} n_{s_{\ba_1}} = n_{s_{a_0}} = \sigma(n_{s_{a_0}}) = \sigma(n_{s_{\ba_0}}) \sigma(n_{s_{\ba_1}}) = n_{s_{\ba_1}} n_{s_{\ba_0}}.\]
Hence $n_{s_{a_0}}^2 = n_{s_{\ba_1}}^2 n_{s_{\ba_0}}^2 = \bb_1^\vee(-1) \bb_0^\vee(-1) \in S_2$ by \S \ref{sec:squares}.  For $i>0$, it is clear anyway from \S \ref{sec:squares} that $n_{s_{a_i}}^2 \in S_2$. In particular, we conclude that a Tits group $\CT_\af^*$ of $W_\af^*$ always exists!

\end{document}